\documentclass[a4paper,reqno,10pt]{amsart}

\usepackage{amssymb}
\usepackage{amsmath}
\usepackage{amsthm}
\usepackage{enumitem}
\usepackage{xcolor}
\usepackage{graphicx}
\usepackage{url}
\usepackage{textcomp}
\usepackage{ytableau}
\usepackage{tikz}

\ytableausetup{nosmalltableaux}

\usepackage{fullpage}

\usepackage{todonotes}

\usepackage{texdraw,amstext,amsfonts,color,tabu,dsfont}

\allowdisplaybreaks

\usepackage{mathtools,tabularx}

\usepackage{float}

\definecolor{foge}{rgb}{0.1, 0.6, 0.1}

\numberwithin{equation}{section}

\newtheorem{theo}{Theorem}[section]

\newtheorem{lem}[theo]{Lemma}
\newtheorem{cor}[theo]{Corollary}
\newtheorem{rem}[theo]{Remark}
\newtheorem{ex}[theo]{Example}

\theoremstyle{definition} 
\newtheorem{deff}[theo]{Definition}

\newcommand{\la}{\lambda}
\newcommand{\ep}{\epsilon}
\newcommand{\Pp}{\mathcal{P}}
\newcommand{\Ppp}{\mathbb{P}}
\newcommand{\Dd}{\mathcal{D}}
\newcommand{\F}{\mathcal{F}}
\newcommand{\E}{\mathcal{E}}

\newcommand{\C}{\mathcal{C}}
\newcommand{\R}{\mathcal{R}}
\newcommand{\Zo}{\mathbb{Z}_{\geq 0}}
\newcommand{\Zu}{\mathbb{Z}_{\geq 1}}

\title{Systematic study of Schmidt-type partitions via weighted words}
\author{Isaac Konan}
\address{Universit\'e de Lyon, Universit\'e Claude Bernard Lyon 1, UMR5208, Institut Camille Jordan, F-69622 Villeurbanne, France}
\email{konan@math.univ-lyon1.fr}

\thanks{This work was supported by the LABEX MILYON (ANR-10-LABX-0070) of Universit\'e de Lyon, within the program ''Investissements d'Avenir" (ANR-11-IDEX-0007) operated by the French National Research Agency (ANR).}

\keywords{Integer partitions, Schmidt-type partitions}

\begin{document}
      
\begin{abstract}
Let $S=(s_n)_{n\geq 1}$ be a sequence with elements in a commutative monoid $(\mathcal{M},+,0)$. In this paper, we provide an explicit formula for
$$\sum_{\la} C(\la) q^{\sum_{n\geq 1} \la_n\cdot s_n}$$
where $\la=(\la_1,\ldots)$ run through some subsets of over-partitions, and $C(\la)$ is a certain product of ``colors'' assigned to the parts of $\la$, and $q^s$ is a formal power of $q$ for $s\in M$. This formula allows us not only to retrieve several known Schmidt-type theorems but also to provide new Schmidt-type theorems for non-periodic sequences $S$. For example, when $(M,+,0)=(\mathbb{Z}_{\geq 0},+,0)$, $s_n=1$ if there exists $i\geq 1$ such  $n=\{i(i-1)/2+1\}$ and $s_n=0$ otherwise, we obtain the following statement: for all non-negative integer $m$, the number of partitions such that $\sum_{i\geq 1}\la_{i(i-1)/2+1} =m$ is equal to the number of plane partitions of $m$. 
Furthermore, we introduce a new family of partitions, the block partitions, generalizing the $k$-elongated partitions. From that family of partitions, we provide a generalization of a Schmidt-type theorem due to Andrews and Paule regarding $k$-elongated partitions and establish a link with the Eulerian polynomials.
\end{abstract}


\maketitle


\section{Introduction}

\subsection{History}

An integer partition is a finite non-increasing sequence of positive integers, i.e $\la = (\la_1,\ldots,\la_s)$ with $\la_1\geq \cdots \geq \la_s >0$. The terms $\la_i$ are called parts of $\la$. We denote by $|\la|=\la_1+\la_2+\cdots+\la_s$ the weight of $\la$. For a non-negative $n$, a partition of $n$ is a partition of weight $n$. For example, the partition of $5$ are $$(5),(4,1),(3,2),(3,1,1),(2,2,1),(2,1,1,1)\text{ and }(1,1,1,1,1).$$ In a 1999 paper \cite{Sc99}, Schmidt proposed the following identity.

\begin{theo}[Schmidt]\label{theo:schmidt}
Let $n$ be a non-negative integer. Then, the number of partitions into distinct parts such that $\la_1+\la_3+\la_5+\cdots=n$ is equal to the number of partitions of $n$. 
\end{theo}

The point of this paper is to investigate ``Schmidt-type'' partitions of $a$,  which, in the idea of Schmidt's identity, are defined as integer partitions $(\la_1,\ldots)$ satisfying $$\sum_{ i\geq 1} \la_i\cdot s_i= a$$
for a certain sequence $S=(s_i)_{i\geq1}$ with elements in a commutative monoid $(\mathcal{M},+,0)$ and an element $a$ in $\mathcal{M}$. Recall that a commutative monoid $(\mathcal{M},+,0)$ is a set containing a neutral element $0$ and a binary operation $+$ such that:
\begin{itemize}
\item[\rm 1)] for all $a,b\in \mathcal{M}$, $b+a=a+b\in \mathcal{M}$,
\item[\rm 1)] for all $a,b,c\in \mathcal{M}$, $(a+b)+c=a+(b+c)$,
\item[\rm 3)] for all $a\in \mathcal{M}$, $a+0=0$,
\end{itemize}
and for all $\ell\in \Zo$ and all $a\in \mathcal{M}$, the element $$\underbrace{a+\cdot+a}_{\ell \text{ times}}$$ is referred to as $\ell \cdot a$.
Following the steps of Schmidt, Uncu, and Andrews--Paule independently provided an analogous theorem for the unrestricted partitions \cite{U18,AP22}.

\begin{theo}[Uncu/Andrews--Paule]\label{theo:ap1}
Let $n$ be a non-negative integer. Then, the number of partitions such that $\la_1+\la_3+\la_5+\cdots=n$ is equal to the number of partitions of $n$ where positive integers appear in two colors. 
\end{theo}

The investigation in the work of Andrews and Paule led to another Schmidt-type theorem related to $k$-elongated partitions, where a link to over-partitions arose. An over-partition is a partition such that positive integers can occur over-lined once. We here adopt the notation where the over-lined occurrence is the last one. For example, the over-partitions of $2$ are $(2),(\overline{2}),(1,1)$ and $(1,\overline{1})$. For a positive integer $k$, a $k$-elongated partition is a finite sequence $(\la_1,\ldots)$ of positive integers such that
$$
\la_1
\geq 
\begin{array}{c}
\la_{2}\\
\la_{3}
\end{array}
\geq \cdots \geq 
\begin{array}{c}
\la_{2k}\\
\la_{2k+1}
\end{array}
\geq 
\la_{2k+2}
\geq 
\begin{array}{c}
\la_{2k+3}\\
\la_{2k+4}
\end{array}
\geq \cdots \geq 
\begin{array}{c}
\la_{4k+1}\\
\la_{4k+2}
\end{array}
\geq 
\la_{4k+3}\geq \cdots .
$$  

The result of Andrews and Paule is then stated as follows.

\begin{theo}[Andrews--Paule]\label{theo:ap2}
Let $n$ be a non-negative integer. Then, the number of $k$-elongated partitions such that $\la_1+\la_{2k+2}+\la_{4k+3}+\cdots=n$ is equal to the number of over-partitions of $n$, where non-over-lined positive integers can occur into $2k+1$ colors, and over-lined integers can occur into $k$ colors. 
\end{theo}

In \cite{BU22}, Brigdes and Uncu proved a refinement of the above theorem for $k=1$ which corresponds to the diamond partitions.
In the aftermath of Andrews and Paule's work, several recent studies related to Schmidt-type theorems were led \cite{J22,LY22,W22}. We here cite an example of such work due to Andrews and Keith. In \cite{AK22}, they explore Schmidt-type partitions where positive integers occur less than $m$ times. In doing so, they provide an identity that generalizes the original Schmidt identity and refines the Glaisher identity \cite{G83}.

\begin{theo}[Andrews--Keith]\label{theo:ak}
Fix $m > 2$.  Let $R = \{ r_1, r_2, \dots, r_i \} \subseteq \{1, 2, \dots, m-1 \}$ with $1 \in R$, and $\vec{\rho} = (\rho_1, \dots, \rho_{m-1})$. Denote by $P_{m,R}(n;\vec{\rho})$ the number of partitions $\lambda = (\lambda_1, \dots, )$ into parts repeating less than $m$ times in which \begin{align*} n &= \sum_{c \equiv r_j \pmod{m}} \lambda_c \\ \rho_k &= \sum_{c \equiv k \pmod{m}} \lambda_k - \lambda_{k+1}.\end{align*}  Then $P_{m,R}(n;\vec{\rho})$ is also equal to the number of partitions of $n$ where parts $k \pmod{i}$ appear in $r_{k+1} - r_k$ colors, letting $r_{i+1} = m$, and, labeling colors of parts $k \pmod{i}$ by $r_k$ through $r_{k+1}-1$, parts of color $j$ appear $\rho_j$ times.
\end{theo}

So far in the literature, the Schmidt-type theorems focus on the sums of parts at indices in a set with periodic gaps, i.e. $S=(s_i)_{i\geq 1}$ periodic with $s_i\in \{0,1\}$. Hence, the goal of this paper is to build a machinery that allows systematic studies of Schmidt-type partitions whatsoever $S$. A further investigation leads to a Schmidt-type theorem on a new family of combinatorial objects, called ``block partitions'', which generalizes the result Andrews--Paule on $k$-elongated partitions.

\subsection{Statement of results}

Let $\la=(\la_i)_{i\geq 1}$ be an infinite sequence of non-negative integers with finitely many positive terms, and where the last occurrence of any positive integers can be over-lined. The terms $\la_i$ are called parts of $\la$ and belong to $\Zo\sqcup \overline{\Zu}$. Furthermore, the over-lined parts are all distinct. Denote by $\Ppp$ the set of such sequences and define the following order on $\Zo\sqcup \overline{\Zu}$:
$$0<\overline{1}<1<\overline{2}<2<\cdots.$$
We also set 
$$\overline{i}\pm \overline{j}=\overline{i}\pm j= i\pm \overline{j}=i\pm j \in \mathbb{Z}.$$ 
This means that the sum (resp. difference) of parts is seen as the sum (resp. difference) of the parts' sizes.
\begin{rem}
For all $u,v\in \Zo\sqcup \overline{\Zu}$ the condition $u-v>0$ is different from the condition $u>v$ unless $u,v$ are both either over-lined or non-over-lined. This comes from the fact that $i-\overline{i}=0$ but $i>\overline{i}$.
\end{rem}
The set of over-partitions $\overline{\Pp}$ consists of non-increasing sequences of $\Ppp$, the set of partitions $\Pp$  consists of over-partitions with non-over-lined parts, and the set $\overline{\Dd}$ consists of over-partitions such that positive parts are over-lined. Finally, let
$\overline{\F}$ be the set of over-partitions $\la$ such that $\overline{i}\in \la$ for all $\overline{i}\leq \la_1$, or equivalently, $\la_i-\la_{i+1}=\chi(\la_i \text{ is over-lined})$.
For all $\la\in \Ppp$, the conjugate of $\la$ is the sequence $\tilde{\la}=(\tilde{\la}_i)_{i\geq 1}$ such that, for all $i\geq 1$, the part $\tilde{\la}_i$ has size
$$\sharp\{j\geq 1: \la_j\geq \overline{i}\}$$
and is over-lined if and only if $\overline{i}$ is a part of $\la$.
The Ferrer diagram of $\la\in \Ppp$ consists of a graphic representation of $\la$ as a succession of rectangles, the $i^{th}$ rectangle being of length $\la_i$ with a marked corner if and only if $\la_i$ is over-lined.

\begin{ex} The sequence $(2,4,\overline{4},\overline{2},3,\overline{1},0,\ldots)$ belongs to $\Ppp$ but not $\overline{\Pp}$, while 
$(4,\overline{4},\overline{3},2,\overline{2},\overline{1},0,\ldots)$ is an over-partition in $\overline{\F}$ but not in $\overline{\Dd}$. Finally, $(4,4,3,2,2,1,0,\ldots)$ is a partition, and $(\overline{4},\overline{2},\overline{1},0,\ldots)$ is an over-partition in $\overline{\Dd}$ and not in $\overline{\F}$. 
The conjugates of $(2,4,\overline{4},\overline{2},3,\overline{1},0,\ldots)$ and $(4,\overline{4},\overline{3},2,\overline{2},\overline{1},0,\ldots)$ are both equal to $(\overline{6},\overline{5},\overline{3},\overline{2},0,\ldots)$.
\end{ex}

\begin{rem}
We retrieve the original notion of over-partition by deleting  the terms equal to $0$ in the aforementioned over-partitions. Inversely, an over-partition as an infinite sequence can be obtained by adding an infinite sequence of parts equal to $0$ to a finite over-partition. The same goes for the partitions.
\end{rem}

\begin{figure}

\begin{center}

\begin{tikzpicture}[scale=0.5, every node/.style={scale=0.7}]

\draw (-5,-3) node {\begin{large}$\tilde{\la}=\tilde{\mu}$\end{large}};


\draw (0,0)--(6,0)--(6,-1)--(5,-1)--(5,-3)--(4,-3)--(4,-4)--(3,-4)--(3,-6)--(1,-6)--(1,-7)--(0,-7)--cycle;

\draw (0,-1)--(5,-1);\draw[<-, dotted] (-0.5,-0.5)--(5.5,-0.5);
\draw (0,-2)--(5,-2);\draw[<-, dotted] (-0.5,-1.5)--(4.5,-1.5);
\draw (0,-3)--(4,-3);\draw[<-, dotted] (-0.5,-2.5)--(4.5,-2.5);
\draw (0,-4)--(3,-4);\draw[<-, dotted] (-0.5,-3.5)--(3.5,-3.5);
\draw (0,-5)--(3,-5);\draw[<-, dotted] (-0.5,-4.5)--(2.5,-4.5);
\draw (0,-6)--(1,-6);\draw[<-, dotted] (-0.5,-5.5)--(2.5,-5.5);
\draw[->] (0,-7)--(0,-9);\draw[<-, dotted] (-0.5,-6.5)--(0.5,-6.5);
\draw[<-, dotted] (-0.5,-7.5)--(0,-7.5);
\draw[<-, dotted] (-0.5,-8.5)--(0,-8.5);

\foreach \x in {1,...,9}
\draw (-1,0.5-\x) node {$\la_{\x}$};

\draw (1,0)--(1,-6);\draw[<-, dotted] (0.5,0.5)--(0.5,-6.5);
\draw (2,0)--(2,-6);\draw[<-, dotted] (1.5,0.5)--(1.5,-5.5);
\draw (3,0)--(3,-4);\draw[<-, dotted] (2.5,0.5)--(2.5,-5.5);
\draw (4,0)--(4,-3);\draw[<-, dotted] (3.5,0.5)--(3.5,-3.5);
\draw (5,0)--(5,-1);\draw[<-, dotted] (4.5,0.5)--(4.5,-2.5);
\draw[->] (6,0)--(8,0);\draw[<-, dotted] (5.5,0.5)--(5.5,-0.5);
\draw[<-, dotted] (6.5,0.5)--(6.5,0);
\draw[<-, dotted] (7.5,0.5)--(7.5,0);

\foreach \x in {1,...,8}
\draw (-0.5+\x,1) node {$\tilde{\la}_{\x}$};

\draw[fill=black] (5,-2.5)--(5,-3)--(4.5,-3)--cycle;
\draw[fill=black] (4,-3.5)--(4,-4)--(3.5,-4)--cycle;
\draw[fill=black] (1,-6.5)--(1,-7)--(0.5,-7)--cycle;


\draw (12,0)--(16,0)--(16,-1)--(17,-1)--(17,-2)--(18,-2)--(18,-3)--(15,-3)--(15,-4)--(12,-4)--(12,-5)--(17,-5)--(17,-6)--(13,-6)--(13,-7)--(15,-7)--(15,-8)--(12,-8)--cycle;

\draw (12,-1)--(16,-1);\draw[<-, dotted] (11.5,-0.5)--(15.5,-0.5);
\draw (12,-2)--(17,-2);\draw[<-, dotted] (11.5,-1.5)--(16.5,-1.5);
\draw (12,-3)--(15,-3);\draw[<-, dotted] (11.5,-2.5)--(17.5,-2.5);
\draw[<-, dotted] (11.5,-3.5)--(14.5,-3.5);
\draw (12,-6)--(13,-6);\draw[<-, dotted] (11.5,-5.5)--(16.5,-5.5);
\draw (12,-7)--(13,-7);\draw[<-, dotted] (11.5,-6.5)--(12.5,-6.5);
\draw[<-, dotted] (11.5,-7.5)--(14.5,-7.5);
\draw[->] (12,-8)--(12,-9);

\foreach \x in {1,...,9}
\draw (11,0.5-\x) node {$\mu_{\x}$};

\draw[->] (16,0)--(19,0);

\foreach \x in {1,...,7}
\draw[<-, dotted] (11.5+\x,0.5)--(11.5+\x,0);
\foreach \x in {1,...,7}
\draw (11.5+\x,1) node {$\tilde{\mu}_{\x}$};

\draw (13,0)--(13,-4); \draw (13,-5)--(13,-6);\draw (13,-7)--(13,-8);
\draw (14,0)--(14,-4); \draw (14,-5)--(14,-6);\draw (14,-7)--(14,-8);
\draw (15,0)--(15,-3); \draw (15,-5)--(15,-6);
\draw (16,-1)--(16,-3); \draw (16,-5)--(16,-6);
\draw (17,-2)--(17,-3);

\draw[fill=black] (16,-0.5)--(16,-1)--(15.5,-1)--cycle;
\draw[fill=black] (17,-5.5)--(17,-6)--(16.5,-6)--cycle;
\draw[fill=black] (13,-6.5)--(13,-7)--(12.5,-7)--cycle;

\end{tikzpicture}

\caption{Ferrers diagrams of $\la=(6,5,\overline{5},\overline{4},3,3,\overline{1},0,\ldots)$ and $\mu=(\overline{4},5,6,3,0,\overline{5},\overline{1},3,0,\ldots)$}

\end{center}

\end{figure}

The next definition introduces the weight in Schmidt-type theorems, i.e., the sum of parts weighed by terms of $S$.

\begin{deff}
Let $S=(s_i)_{i\geq 1}$ be a sequence of elements in a commutative monoid $\mathcal{M}$. The $S$-weight $|\cdot|_S$ is the function from $\Ppp$ to $\mathcal{M}$ defined by
\begin{equation}\label{eq:sweight}
\la \mapsto |\la|_S=\sum_{i\geq 1} \la_i\cdot s_i.
\end{equation}
By convention, we set $\overline{j}\cdot s_i = j\cdot s_i$ for all $j\in \Zo$ and $i\geq 1$. We also define the sequence $(S_{i})_{i\in \Zo\sqcup \overline{\Zu}}$ such that
\begin{equation}\label{eq:sgif}
S_i = \sum_{0<\overline{j}\leq i}s_j.
\end{equation}
\end{deff}

\begin{figure}

\begin{center}

\begin{tikzpicture}[scale=0.5, every node/.style={scale=0.7}]


\draw (0,0)--(6,0)--(6,-1)--(5,-1)--(5,-3)--(4,-3)--(4,-4)--(3,-4)--(3,-6)--(1,-6)--(1,-7)--(0,-7)--cycle;

\draw (0,-1)--(5,-1);\draw[<-, dotted,red] (-0.5,-0.5)--(5.5,-0.5);
\draw (0,-2)--(5,-2);\draw[<-, dotted] (-0.5,-1.5)--(4.5,-1.5);
\draw (0,-3)--(4,-3);\draw[<-, dotted,red] (-0.5,-2.5)--(4.5,-2.5);
\draw (0,-4)--(3,-4);\draw[<-, dotted] (-0.5,-3.5)--(3.5,-3.5);
\draw (0,-5)--(3,-5);\draw[<-, dotted,red] (-0.5,-4.5)--(2.5,-4.5);
\draw (0,-6)--(1,-6);\draw[<-, dotted] (-0.5,-5.5)--(2.5,-5.5);
\draw[->] (0,-7)--(0,-9);\draw[<-, dotted,red] (-0.5,-6.5)--(0.5,-6.5);
\draw[<-, dotted] (-0.5,-7.5)--(0,-7.5);
\draw[<-, dotted,red] (-0.5,-8.5)--(0,-8.5);

\foreach \x in {1,3,5,7,9}
\draw (-1,0.5-\x) node[red] {$\la_{\x}$};

\foreach \x in {2,4,6,8}
\draw (-1,0.5-\x) node {$\la_{\x}$};`

\draw (1,0)--(1,-6);
\draw (2,0)--(2,-6);
\draw (3,0)--(3,-4);
\draw (4,0)--(4,-3);
\draw (5,0)--(5,-1);
\draw[->] (6,0)--(8,0);

\draw[fill=black] (5,-2.5)--(5,-3)--(4.5,-3)--cycle;
\draw[fill=black] (4,-3.5)--(4,-4)--(3.5,-4)--cycle;
\draw[fill=black] (1,-6.5)--(1,-7)--(0.5,-7)--cycle;


\draw (12,0)--(16,0)--(16,-1)--(17,-1)--(17,-2)--(18,-2)--(18,-3)--(15,-3)--(15,-4)--(12,-4)--(12,-5)--(17,-5)--(17,-6)--(13,-6)--(13,-7)--(15,-7)--(15,-8)--(12,-8)--cycle;

\draw (12,-1)--(16,-1);\draw[<-, dotted,red] (11.5,-0.5)--(15.5,-0.5);
\draw (12,-2)--(17,-2);\draw[<-, dotted] (11.5,-1.5)--(16.5,-1.5);
\draw (12,-3)--(15,-3);\draw[<-, dotted,red] (11.5,-2.5)--(17.5,-2.5);
\draw[<-, dotted] (11.5,-3.5)--(14.5,-3.5);
\draw (12,-6)--(13,-6);\draw[<-, dotted] (11.5,-5.5)--(16.5,-5.5);
\draw (12,-7)--(13,-7);\draw[<-, dotted,red] (11.5,-6.5)--(12.5,-6.5);
\draw[<-, dotted] (11.5,-7.5)--(14.5,-7.5);
\draw[->] (12,-8)--(12,-9);
\draw[<-, dotted,red] (11.5,-4.5)--(12,-4.5);
\draw[<-, dotted,red] (11.5,-8.5)--(12,-8.5);

\foreach \x in {1,3,5,7,9}
\draw (11,0.5-\x) node[red] {$\mu_{\x}$};

\foreach \x in {2,4,6,8}
\draw (11,0.5-\x) node {$\mu_{\x}$};

\draw[->] (16,0)--(19,0);

\draw (13,0)--(13,-4); \draw (13,-5)--(13,-6);\draw (13,-7)--(13,-8);
\draw (14,0)--(14,-4); \draw (14,-5)--(14,-6);\draw (14,-7)--(14,-8);
\draw (15,0)--(15,-3); \draw (15,-5)--(15,-6);
\draw (16,-1)--(16,-3); \draw (16,-5)--(16,-6);
\draw (17,-2)--(17,-3);

\draw[fill=black] (16,-0.5)--(16,-1)--(15.5,-1)--cycle;
\draw[fill=black] (17,-5.5)--(17,-6)--(16.5,-6)--cycle;
\draw[fill=black] (13,-6.5)--(13,-7)--(12.5,-7)--cycle;

\end{tikzpicture}

\caption{$(1,0,1,0,1,\ldots)$-weights $\la=(6,5,\overline{5},\overline{4},3,3,\overline{1},0,\ldots)$ and $\mu=(\overline{4},5,6,3,0,\overline{5},\overline{1},3,0,\ldots)$}

\end{center}

\end{figure}

The weighted words consist in assigning some colors to the parts. In this spirit, we define a set of colors and assigned them to the elements of $\mathcal{M}$ under certain conditions.

\begin{deff}
Let $\C=\{c_i:i\in \Zo\}\cup \{c_{\overline{i}}:i\in \Zu\}$ be a set of colors. For all $\la\in \Ppp$,  set $M(\la)=\max\{\la_i:i\geq 1\}$ and let \textit{the weight and color sequence} of $\la$ be defined respectively by $$|\la| = \sum_{i\geq 1} |\la_i| \quad\text{and}\quad C(\la)=\prod_{0<\overline{i}\leq M(\la)} c_{\tilde{\la}_i}.$$
By convention, $C((0,\ldots))=1$. The assignment of the colors is as follows. The color $c_0$ can only be assigned to the neutral element $0$ of $\mathcal{M}$, and for all $i\in \Zo\sqcup \overline{\Zu}$, the colors $c_i$ can only be assigned to $S_i$.
For $k_c$ a colored part, we write $|k_c|=k$ and $c(k_c)$, called respectively the size and the color of $k_c$.

For $Ens \subseteq \overline{\Pp}$, denote by $Ens_\C^S$ the set of sequences $\la=(\la_i)_{i\geq 1}$ into colored parts, such that the sequence $(j_i)_{i\geq 1}$ belongs to $Ens$, where $c_{j_i}=c(\la_i)$ for all $i\geq 1$. For a colored partition $\la$, we define the weight and the color sequence of $\la$ to be respectively 
$$|\la| = \sum_{i\geq 1} |\la_i| \quad\text{and}\quad C(\la)=\prod_{\substack{i\geq 1\\\la_i\neq 0_{c_0}
}}c(\la_i).$$ 
\end{deff}

\begin{ex}For $\mathcal{M}=\Zo$ and $S=(\chi(i \text{ odd}))_{i\geq 1}$,
the positive integers $i$ can only be colored by $c_j$ with $j\in \{2i-1,\overline{2i-1},2i,\overline{2i}\}$. Hence, $(4_{c_8},4_{c_8},4_{c_{\overline{8}}},3_{c_{\overline{5}}},2_{c_{4}},1_{c_1},1_{c_1},1_{c_{\overline{1}}},0_{c_0},\ldots)$ is in $\overline{\Pp}_\C^S$, but not  $(4_{c_8},4_{c_{\overline{8}}},4_{c_8},3_{c_{\overline{5}}},2_{c_{4}},1_{c_1},1_{c_{\overline{1}}},1_{c_{\overline{1}}},0_{c_0},\ldots)$.
\end{ex}

The main result of the paper is a Schmidt-type theorem connecting the $S$-weight, the conjugate of an over-partition, and the $S$ greatest integer function. 

\begin{theo}\label{theo:main}
Let $\Phi_S:\la \mapsto (\mu_i)_{i\geq 1}$ such that, for all $i\geq 1$, $\mu_i$ has size $S_{\tilde{\la}_i}$ and color $c_{\tilde{\la}_i}$. Then, 
$$|\Phi_S(\la)|=|\la|_S$$ and $C(\Phi_S(\la))=C(\la)$. Moreover, $\Phi_S$ describes a bijection from $\overline{\Pp}$ to $\overline{\Pp}_\C^S$, from $\Pp$ to $\Pp_\C^S$, from $\overline{\F}$ to $\overline{\Dd}_\C^S$, and from $\overline{\Dd}$ to $\overline{\F}_\C^S$, and we have 

\begin{align}
\sum_{\la\in \overline{\Pp}} C(\la) q^{|\la|_S} &=  \prod_{i\geq 1}\frac{1+c_{\overline{i}} q^{S_i}}{1-c_i q^{S_i}},\label{eq:ovp}\\
\sum_{\la\in \Pp} C(\la) q^{|\la|_S} &=  \prod_{i\geq 1}\frac{1}{1-c_i q^{S_i}},\label{eq:p}\\
\sum_{\la\in \overline{\F}} C(\la) q^{|\la|_S} &= \prod_{i\geq 1}(1+c_{\overline{i}} q^{S_i}).\label{eq:ovf}\\
\sum_{\la\in \overline{\Dd}} C(\la) q^{|\la|_S} &= \sum_{\overline{i}>0} \prod_{0<\overline{j}<\overline{i}}
\frac{c_{\overline{j}} q^{S_j}}{1-c_jq^{S_j}}.
\label{eq:ovd}
\end{align}

\end{theo}

\begin{rem}
We observe that for $\la\in \overline{\Pp}$, $M(\la)=\la_1$. Hence, in the above theorem, by replacing $c_i$ by $zc_i$, the power of $z$ will then count the parameter $\la_1$.
\end{rem}

For $\mathcal{M}=\Zo$ and $S=(\chi(i \text{ odd}))_{i\geq 1}$, we retrieve Theorem \ref{theo:ap1} from the identity \eqref{eq:p}. We derive from Theorem \ref{theo:main} several results for $\mathcal{M}=\Zo$ and non-periodic $S$ and some of them involve the notion of plane partitions.
A plane partition can be seen as a sequence $(\la_{i,j})_{i,j\geq 1}$ of non-negative integers, with finitely many positive terms, and such that the sequences $(\la_{i,j'})_{j'\geq 1}$ and $(\la_{i',j})_{i'\geq 1}$ are non-increasing for all $i,j\geq 1$. The weight of a plane partition is $\sum_{i,j\geq 1} \la_{i,j}$, and the generating function of plane partitions in terms of the weight is provided by Mac-Mahon \cite{Mac16} and equals 
$$\prod_{n\geq 1}\frac{1}{(1-q^n)^n}.$$
The $q$-Pochhammer symbol is defined  for $n \in \Zo \cup \{\infty\}$ as
\begin{align*}
  (a;q)_n &:= \prod_{k=0}^{n-1} (1-aq^k).
\\ (a_1, \dots, a_j ; q)_n &:= (a_1;q)_n \cdots (a_j;q)_n.
\end{align*}
We finally set $\chi(A)=1$ if the proposition $A$ is true and $\chi(A)=0$ if not.
\begin{cor}
We have the following.
\begin{enumerate}
\item For $S$ such that $s_i=\chi(\exists n\geq 1\,, i=n(n-1)/2+1)$, the number of partitions of $S$-weight $m$ is equal to the number of plane partitions of weight $m$.
The corresponding identity is 
$$\sum_{\la\in \Pp} q^{|\la|_S}  = \prod_{n\geq 1} \frac{1}{(1-q^n)^n}.$$
\item For $S$ such that $s_i=\chi(\exists n\geq 1\,, i=n(n+1)/2)$,, the number of partitions of $S$-weight $m$ is equal to the number of pairs $(\mu,\nu)$, where $\mu$ is a partition and $\nu$ is a plane partition, whose weight sum is $m$.
The corresponding identity is 
$$\sum_{\la\in \Pp} q^{|\la|_S}  = \frac{1}{(q;q)_\infty}\cdot\prod_{n\geq 1} \frac{1}{(1-q^n)^n}.$$
\item For $S$ such that $s_i=\chi(\exists n\geq 1\,, i=n^2)$,, the number of partitions of $S$-weight $m$ is equal to the number of triplets $(\mu,\nu,\kappa)$, where $\mu$ is a partition and $\nu,\kappa$ are plane partitions, whose weight sum is $m$.
The corresponding identity is 
$$\sum_{\la\in \Pp} q^{|\la|_S}  = \frac{1}{(q;q)_\infty}\cdot\left(\prod_{n\geq 1} \frac{1}{(1-q^n)^n}\right)^2.$$
\item Recall the Fibonacci sequence $F_{n+2}=F_{n+1}+F_n$ for all $n\in \Zo$ with $F_0=0$ and $F_1=1$. For $S$ such that $s_i=\chi(\exists n\geq 1\,, i=F_{n+1})$, the number of over-partitions of $S$-weight $m$ is equal to the number of over-partitions of weight $m$ where the integer $i$ appears in $F_i$ colors.
The corresponding identity is 
$$\sum_{\la\in \Pp} C(\la) q^{|\la|_S}  = \prod_{n\geq 1} \left(\frac{1+q^n}{1-q^n}\right)^{F_n}.$$
\item For $S$ such that $s_i=\chi(\exists n\geq 1\,, i=2^{n-1})$, the number of over-partitions of $S$-weight $m$ is equal to the number of over-partitions of weight $m$ where the integer $i$ appears in $2^{i-1}$ colors.
The corresponding identity is 
$$\sum_{\la\in \Pp} q^{|\la|_S}  = \prod_{n\geq 1}\left(\frac{1+q^n}{1-q^n}\right)^{2^{n-1}}.$$
\end{enumerate}
\end{cor}

\subsubsection{Unrestricted Schimdt-type partitions modulo $m$}

We now return to the case where $S$ is periodic of period $m\geq 1$. Set for all $u\in \Zu$, set $c_u=c_{u+m}$ and $c_{\overline{u}}=c_{\overline{u+m}}$. We also define for all $\la =(\la_i)_{i\geq 1}\in \overline{\Pp}$ and $1\leq j\leq m$
\begin{align*}
\rho_j(\la) &= \sum_{i\geq 0} \la_{j+im}-\la_{j+1+im}-\chi(\la_{j+im}\text{ is over-lined})\\
\rho_{\overline{j}}(\la) &= \sharp\{i\geq 0: \la_{j+im}\text{ is over-lined}\}.
\end{align*}
The result corresponding to Theorem \ref{theo:main} is then the following.

\begin{theo}\label{theo:unrestrictedmodm}
We have
\begin{align}
\sum_{\la\in \overline{\Pp}} \left(\prod_{j=1}^m c_j^{\rho_j(\la)}c_{\overline{j}}^{\rho_{\overline{j}}(\la)} \right)q^{|\la|_S} &= \frac{(-c_{\overline{1}} q^{S_1},\ldots,-c_{\overline{m}} q^{S_m};q^{S_m})_\infty}{(c_{1} q^{S_1},\ldots,c_{m} q^{S_m};q^{S_m})_\infty},\label{eq:ovpmodm}\\
\sum_{\la\in \Pp} \left(\prod_{j=1}^m c_j^{\rho_j(\la)} \right) q^{|\la|_S} &= \frac{1}{(c_{1} q^{S_1},\ldots,c_{m} q^{S_m};q^{S_m})_\infty},\label{eq:pmodm}
\\
\sum_{\la\in \overline{\F}} \left(\prod_{j=1}^m c_{\overline{j}}^{\rho_{\overline{j}}(\la)} \right)q^{|\la|_S} &= (-c_{\overline{1}} q^{S_1},\ldots,-c_{\overline{m}} q^{S_m};q^{S_m})_\infty,\label{eq:ovfmodm}
\\
\sum_{\la\in \overline{\Dd}} \left(\prod_{j=1}^m c_j^{\rho_j(\la)}c_{\overline{j}}^{\rho_{\overline{j}}(\la)} \right)q^{|\la|_S} &=\sum_{n\geq 0} \sum_{\ell=1}^m \frac{(c_{\overline{1}}\cdots c_{\overline{\ell-1}})^{n+1} q^{(n+1)(S_1+\cdots+S_{\ell-1})+\frac{(\ell-1)n(n+1)}{2}S_m}}{(c_{1} q^{S_1},\ldots,c_{\ell-1} q^{S_{\ell-1}};q^{S_m})_{n+1}}\nonumber\\
&\qquad\qquad\qquad\cdot\frac{(c_{\overline{\ell}}\cdots c_{\overline{m}})^{n} q^{n(S_{\ell}+\cdots+S_{m})+\frac{(m-\ell+1)n(n-1)}{2}S_m}}{(c_{\ell} q^{S_{\ell}},\ldots,c_{m} q^{S_m};q^{S_m})_{n}}.\label{eq:ovdmodm}
\end{align}
\end{theo}

We here give an example with $\mathcal{M}=(\Zo)^{m+1}$ and $$s_i=(1,\chi(i\equiv 1 \mod m),\chi(i\equiv 2 \mod m),\ldots,\chi(i\equiv m \mod m))$$
for all $i\in \Zu$.
We also define parameters $q_0,\ldots,q_m$ and the power of $q$ such that
$$q^{(u_0,\ldots,u_m)}=q_0^{u_0}q_1^{u_1}\cdots q_m^{u_m}$$
for all $(u_0,\ldots,u_m)\in (\Zo)^{m+1}$.
\begin{cor}
We have
\begin{align*}
\sum_{\la\in \overline{\Pp}} \left(\prod_{j=1}^m c_j^{\rho_j(\la)}c_{\overline{j}}^{\rho_{\overline{j}}(\la)} \right)q_0^{|\la|}\prod_{i=1}^m q_i^{\sum_{j\geq 0}\la_{i+jm}} &= \prod_{i=1}^m\frac{(-c_{\overline{i}}q_0^i q_1\cdots q_i;q_0^m q_1\cdots q_m)_\infty}{(c_{i}q_0^i q_1\cdots q_i;q_0^m q_1\cdots q_m)_\infty},\\
\sum_{\la\in \Pp} \left(\prod_{j=1}^m c_j^{\rho_j(\la)}\right) q_0^{|\la|}\prod_{i=1}^m q_i^{\sum_{j\geq 0}\la_{i+jm}} &= \prod_{i=1}^m\frac{1}{(c_{i}q_0^i q_1\cdots q_i;q_0^m q_1\cdots q_m)_\infty},
\\
\sum_{\la\in \overline{\F}} \left(\prod_{j=1}^m c_{\overline{j}}^{\rho_{\overline{j}}(\la)} \right)q_0^{|\la|}\prod_{i=1}^m q_i^{\sum_{j\geq 0}\la_{i+jm}} &= \prod_{i=1}^m(-c_{\overline{i}}q_0^i q_1\cdots q_i;q_0^m q_1\cdots q_m)_\infty,
\\
\sum_{\la\in \overline{\Dd}} \left(\prod_{j=1}^m c_j^{\rho_j(\la)}c_{\overline{j}}^{\rho_{\overline{j}}(\la)} \right)q_0^{|\la|}\prod_{i=1}^m q_i^{\sum_{j\geq 0}\la_{i+jm}} &=\sum_{n\geq 0} \sum_{\ell=1}^m \prod_{i=1}^{\ell-1}\frac{(c_{\overline{i}}q_0^i q_1\cdots q_i)^{n+1}\cdot (q_0^m q_1\cdots q_m)^{\frac{n(n+1)}{2}}}{(c_{i}q_0^i q_1\cdots q_i;q_0^m q_1\cdots q_m)_{n+1}}\\
&\qquad\qquad\cdot \prod_{i=\ell}^{m}\frac{(c_{\overline{i}}q_0^i q_1\cdots q_i)^{n}\cdot (q_0^m q_1\cdots q_m)^{\frac{n(n-1)}{2}}}{(c_{i}q_0^i q_1\cdots q_i;q_0^m q_1\cdots q_m)_{n}}.
\end{align*}
\end{cor}

\subsubsection{Restricted Schimdt-type partitions modulo $m$}
Under the assumptions of the previous section, we now shift our focus on new subsets of over-partitions with some restrictions on the number of occurrences of non-over-lined parts.
Let $\overline{\Pp}_m$ be the set of over-partitions with less than $m$ non-over-lined occurrences for all positive integers, 
let $\Pp_m$ be the set of partitions of $\overline{\Pp}_m$, and let $\overline{\F}_m$ be the subset of $\overline{\Pp}_m$ consisting of over-partitions such that $\la_i-\la_{i+1}=\chi(\la_i \text{ is over-lined})$ for all $i$ not divisible by $m$.
The critical result related to sets with periodic gaps is given by the following theorem.

\begin{theo}\label{theo:restrictedmodm}
By setting $c_m=1$, we have

\begin{align}
\sum_{\la\in \overline{\Pp}_m} \left(\prod_{j=1}^{m-1} c_j^{\rho_j(\la)}c_{\overline{j}}^{\rho_{\overline{j}}(\la)} \right) c_{\overline{m}}^{\rho_{\overline{m}}(\la)} q^{|\la|_S} &= \frac{(-c_{\overline{1}} q^{S_1},\ldots,-c_{\overline{m}} q^{S_m};q^{S_m})_\infty}{(c_{1} q^{S_1},\ldots,c_{m-1} q^{S_{m-1}};q^{S_m})_\infty},\label{eq:restrictedovpmodm}\\
\sum_{\la\in \Pp_m} \left(\prod_{j=1}^{m-1} c_j^{\rho_j(\la)} \right) q^{|\la|_S} &= \frac{1}{(c_{1} q^{S_1},\ldots,c_{m-1} q^{S_{m-1}};q^{S_m})_\infty},
\label{eq:restrictedpmodm}\\
\sum_{\la\in \overline{\F}_m} \left(\prod_{j=1}^m c_{\overline{j}}^{\rho_{\overline{j}}(\la)} \right)q^{|\la|_S} &= (-c_{\overline{1}} q^{S_1},\ldots,-c_{\overline{m}} q^{S_m};q^{S_m})_\infty.\label{eq:restrictedovdmodm}
\end{align}
\end{theo}

In particular, for $\mathcal{M}=\Zo$ and $s_i=\chi(\exists n\in \{1,\ldots,t\}\,,\, i=r_n)$ for $1\leq i\leq m$, the identity \eqref{eq:restrictedpmodm} is equivalent to Theorem \ref{theo:ak}. From \eqref{eq:ovfmodm} and \eqref{eq:restrictedovdmodm}, one can derive the following identity.
\begin{cor}
Suppose that $c_m=c_0=1$. Then, at fixed $S$-weight and color sequence, there are as many over-partitions in $\overline{\F}$ as over-partitions in $\overline{\F}_m$.
\end{cor}

For $\mathcal{M}=(\Zo)^{m+1},$
$$s_i=(1,\chi(i\equiv 1 \mod m),\chi(i\equiv 2 \mod m),\ldots,\chi(i\equiv m \mod m))$$
for all $i\in \Zu$,
and parameters $q_0,\ldots,q_m$ such that
$$q^{(u_0,\ldots,u_m)}=q_0^{u_0}q_1^{u_1}\cdots q_m^{u_m}$$
for all $(u_0,\ldots,u_m)\in (\Zo)^{m+1}$,
we have the following corollary to Theorem \ref{theo:restrictedmodm}. 
\begin{cor}
We have
\begin{align*}
\sum_{\la\in \overline{\Pp}_m}  \left(\prod_{j=1}^{m-1} c_j^{\rho_j(\la)}c_{\overline{j}}^{\rho_{\overline{j}}(\la)} \right)c_{\overline{m}}^{\rho_{\overline{m}}(\la)}q_0^{|\la|}\prod_{i=1}^m q_i^{\sum_{j\geq 0}\la_{i+jm}} &= \prod_{i=1}^{m-1}\frac{(-c_{\overline{i}}q_0^i q_1\cdots q_i;q_0^m q_1\cdots q_m)_\infty}{(c_{i}q_0^i q_1\cdots q_i;q_0^m q_1\cdots q_m)_\infty}\\
&\qquad\times (-c_{\overline{m}}q_0^m q_1\cdots q_m;q_0^m q_1\cdots q_m)_\infty\\
\sum_{\la\in \Pp_m} \left(\prod_{j=1}^{m-1} c_j^{\rho_j(\la)}\right) q_0^{|\la|}\prod_{i=1}^m q_i^{\sum_{j\geq 0}\la_{i+jm}} &= \prod_{i=1}^{m-1}\frac{1}{(c_{i}q_0^i q_1\cdots q_i;q_0^m q_1\cdots q_m)_\infty},
\\
\sum_{\la\in \overline{\F}_m} \left(\prod_{j=1}^m c_{\overline{j}}^{\rho_{\overline{j}}(\la)} \right)q_0^{|\la|}\prod_{i=1}^m q_i^{\sum_{j\geq 0}\la_{i+jm}} &= \prod_{i=1}^m(-c_{\overline{i}}q_0^i q_1\cdots q_i;q_0^m q_1\cdots q_m)_\infty.
\end{align*}
\end{cor}

Using \eqref{eq:restrictedpmodm}, we obtain the following identity.

\begin{cor}
Let $d,d_1,\ldots,d_t$ be positive integers such that $d_1\leq\cdots\leq d_t\leq d$, and set $d_0=0$ and $d_{t+1}=d$. Then, for all non-negative integer $n$, the number of partitions $n$ into parts congruent to $d_1,\ldots,d_t \mod d$ such parts congruent to $d_i\mod d$ are indexed by $i$, is equal to the number of partitions $(\lambda_1,\ldots)$ into parts repeating at most $t$ times, and such that
$$n=\sum_{i=1}^{t+1}(d_{i}-d_{i-1})\sum_{j\equiv i \mod t+1} \la_j.$$
The corresponding $q$-series is
$$\sum_{\la\in \Pp_{t+1}} q^{\sum_{i=1}^{t+1}(d_{i}-d_{i-1})\sum_{j\equiv i \mod t+1} \la_j} = \frac{1}{(q^{d_1},\ldots,q^{d_t};q^d)_\infty}.$$
\end{cor} 
We thus derive the following companions of the Andrews--Bressoud--Gordon identities \cite{A74,B80}.

\begin{theo}[Companions of the Andrews--Bressoud--Gordon identities]
Let $i,k,j$ be non-negative integers such that $j\in \{0,1\}$ and $1\leq i\leq k+j-1$. Then, for all non-negative integer $n$, the three following sets of partitions are equinumerous:
\begin{itemize}
\item[\rm 1)] the set of partitions of $n$ into parts not congruent to $0,\pm i \mod 2k+j$,
\item[\rm 2)] the set of partitions $(\la_1,\ldots)$ of $n$ with at most $i-1$ parts equal to $1$, and such that $\la_t-\la_{t+k-1}\geq 2$ for all $t\geq 1$, with the additional condition that if $\la_t-\la_{t+k-2}\leq 1$, then $$\la_t+\cdots+\la_{t+k-2}\equiv i-1 \mod 2-j,$$
\item[\rm 3)] the set of partitions $(\la_1,\ldots)$ of $n$ such that $\la_t-\la_{t+2k+j-3}\geq 1$ for all $t\geq 1$, and 
$$n=\sum_{t\geq 1}(1+\chi(t\equiv i)+\chi(t\equiv 2k+j-i-1))\la_t$$
with the congruence taken modulo $2k+j-2$.
\end{itemize}
\end{theo}

\subsubsection{Block partitions}

Let $U=(u_n)_{n\geq 1}$ be a sequence of positive integers and set $U_n = \sum_{i=1}^n u_n$ for all $n\geq 0$. 
\begin{deff}
A \textit{block over-partition} of type $U$ is an infinite sequence $\la=(\la_i)_{i\geq 1}$ in $\Ppp$ such that, for all $n\geq 1$, $\min\{\la_{U_{n-1}+1},\ldots,\la_{U_n}\}\geq \max\{\la_{U_n+1},\ldots,\la_{U_{n+1}}\}$, i.e.,

$$
\begin{array}{c}
\la_1\\
\vdots\\
\la_{U_1}
\end{array}
\geq 
\begin{array}{c}
\la_{U_1+1}\\
\vdots\\
\la_{U_2}
\end{array}
\geq \cdots \geq 
\begin{array}{c}
\la_{U_{n-1}+1}\\
\vdots\\
\la_{U_{n}}
\end{array}
\geq 
\begin{array}{c}
\la_{U_{n}+1}\\
\vdots\\
\la_{U_{n+1}}
\end{array}
\geq \cdots
$$
Let $\overline{\Pp}(U)$ be the set of block over-partitions of type $U$, and let the set of block partitions of type $U$ be $\Pp(U)=\overline{\Pp}(U)\cap \Pp$. Finally, let $\overline{\Dd}(U)$ be the set of over-partitions of $\overline{\Pp}(U)$ such that, for all $n\geq 1$, $\min\{\la_{U_{n-1}+1},\ldots,\la_{U_n}\}-\max\{\la_{U_n+1},\ldots,\la_{U_{n+1}}\}>0$ when $\min\{\la_{U_{n-1}+1},\ldots,\la_{U_n}\}>0$.
\end{deff}
The aforementioned objects generalize the $k$-elongated partitions, which correspond to the block partitions of type $U$ for $u_{(k+1)i+1}=1$ and $u_{(k+1)i+1+j}=2$ for all $i\geq 1$ and $1\leq j\leq k$.
To understand the combinatorics behind the block over-partitions, we first need to comprehend their link to the symmetric group.

\begin{deff}
For a positive integer $n$, let $\mathcal{S}_n$ be the corresponding symmetric group, i.e. the set of permutations of $\{1,\ldots,n\}$. Define the polynomials
$$\overline{\E}_n (x_1,y_1,x_2,y_2,\ldots,x_{n-1},y_{n-1},x_n,y_n) = (y_n+x_n)\sum_{\sigma\in \mathcal{S}_n} \prod_{i=1}^{n-1}(y_i^{\chi(\sigma(i)>\sigma(i+1))}+x_i),$$
and 
$$\E_n (y_1,y_2,\ldots,y_{n-1})=\overline{\E}_n (0,y_1,0,y_2,\ldots,0,y_{n-1},0,1) = \sum_{\sigma\in \mathcal{S}_n}\prod_{i=1}^{n-1}y_i^{\chi(\sigma(i)>\sigma(i+1))}.$$
\end{deff}
Note that $\overline{\E}_n (0,x,0,x,\ldots,0,x,0,x)$ is the Eulerian polynomial  and satisfies  
$$\overline{\E}_n (0,x,0,x,\ldots,0,x,0,x) = (1-x)^{n+1}\sum_{m\geq 0} m^n x^m$$
(see Proposition 1.4.4 in \cite{St11}), while 
$$\overline{\E}_n (0,q,0,q^2,\ldots,0,q^{n-1},0,1) = \frac{(q;q)_n}{(1-q)^n}$$
(see Proposition 1.4.6 in \cite{St11}). For example, 
$$\overline{\E}_1(x_1)=y_1+x_1,\quad\E_1=1,$$
$$\overline{\E}_2(x_1,y_1,x_2)=(y_2+x_2)[(1+x_1)+(y_1+x_1)],\quad\E_2(y_1)=1+y_1,$$
$$\overline{\E}_3(x_1,y_1,x_2,y_2,x_3)=(y_3+x_3)[(1+x_1)(1+x_2)+2(y_1+x_1)(1+x_2)+2(1+x_1)(y_2+x_2)+(y_1+x_1)(y_2+x_2)],$$
$$\E_3(y_1,y_2)=1+2y_1+2y_2+y_1y_2.$$
The following theorem then gives the generating function of the block over-partitions of type $U$. 
\begin{theo}\label{theo:pu}
We have 
\begin{align}
\sum_{\la\in \overline{\Pp}(U)}C(\la)q^{|\la|}&= \prod_{n\geq 1}\frac{\overline{\E}_{u_n}(c_{\overline{U_{n-1}+1}}q^{U_{n-1}+1},c_{U_{n-1}+1}q^{U_{n-1}+1},\ldots,c_{\overline{U_{n}-1}}q^{U_{n}-1},c_{U_{n}-1}q^{U_{n}-1},c_{\overline{U_{n}}}q^{U_{n}},1)}{(1-c_{U_{n-1}+1}q^{U_{n-1}+1})\cdots(1-c_{U_{n}}q^{U_{n}})},\label{eq:ovpu}\\
\sum_{\la\in \Pp(U)}C(\la)q^{|\la|}&= \prod_{n\geq 1}\frac{\E_{u_n}(c_{U_{n-1}+1}q^{U_{n-1}+1},\ldots,c_{U_{n}-1}q^{U_{n}-1})}{(1-c_{U_{n-1}+1}q^{U_{n-1}+1})\cdots(1-c_{U_{n}}q^{U_{n}})},\label{eq:pu}\\
\sum_{\la\in \overline{\Dd}(U)}C(\la)q^{|\la|}&= \sum_{n\geq 1} \frac{\overline{\E}_{u_n}(c_{\overline{U_{n-1}+1}}q^{U_{n-1}+1},c_{U_{n-1}+1}q^{U_{n-1}+1},\ldots,c_{\overline{U_{n}-1}}q^{U_{n}-1},c_{U_{n}-1}q^{U_{n}-1},0,1)}{(1-c_{U_{n-1}+1}q^{U_{n-1}+1})\cdots(1-c_{U_{n}-1}q^{U_{n}-1})}\nonumber\\
&\qquad\cdot \prod_{m=1}^{n-1}\frac{\overline{\E}_{u_m}(c_{\overline{U_{m-1}+1}}q^{U_{m-1}+1},c_{U_{m-1}+1}q^{U_{m-1}+1},\ldots,c_{\overline{U_{m}}}q^{U_{m}},c_{U_{m}}q^{U_{m}})}{(1-c_{U_{m-1}+1}q^{U_{m-1}+1})\cdots(1-c_{U_{m}}q^{U_{m}})}.
\label{eq:ovdu}
\end{align}
\end{theo}
Our goal is to state a generalization of Theorem \ref{theo:ap2}
in Schmidt's spirit. To do so, we consider an increasing sequence $(n_i)_{i\geq 1}$ of positive integers, finite or not, such that $u_{n_i}=1$ for all $1\leq i$. Set $n_0=0$, $\mathcal{M}=\Zo$ and $s_i=\chi(\exists j\geq 1,\, i=U_{n_j})$. We then have the following.
\begin{theo}\label{theo:pus}
We have 
\begin{align}
\sum_{\la\in \overline{\Pp}(U)}C(\la)q^{|\la|_S}&= \prod_{i\geq 1}\frac{1+c_{\overline{n_i}}q^i}{1-c_{n_i}q^{i}}\nonumber\\&\qquad\cdot \prod_{m=n_{i-1}+1}^{n_i-1}\frac{\overline{\E}_{u_m}(c_{\overline{U_{m-1}+1}}q^{i-1},c_{U_{m-1}+1}q^{i-1},\ldots,c_{\overline{U_{m}-1}}q^{i-1},c_{U_{n}-1}q^{i-1},c_{\overline{U_{n}}}q^{i-1},1)}{(1-c_{U_{n-1}+1}q^{i-1})\cdots(1-c_{U_{n}}q^{i-1})},
\label{eq:ovpus}\\
\sum_{\la\in \Pp(U)}C(\la)q^{|\la|_S}&= \prod_{i\geq 1}\frac{1}{1-c_{n_i}q^{i}}\cdot \prod_{m=n_{i-1}+1}^{n_i-1}\frac{\E_{u_m}(c_{U_{m-1}+1}q^{i-1},\ldots,c_{U_{n}-1}q^{i-1})}{(1-c_{U_{n-1}+1}q^{i-1})\cdots(1-c_{U_{n}}q^{i-1})},\label{eq:pus}\\
\sum_{\la\in \overline{\Dd}(U)}C(\la)q^{|\la|_S}&= \sum_{i\geq 1}
\sum_{k=n_{i-1}+1}^{n_i-1} \frac{\overline{\E}_{u_k}(c_{\overline{U_{k-1}+1}}q^{i-1},c_{U_{k-1}+1}q^{i-1},\ldots,c_{\overline{U_{k}-1}}q^{i-1},c_{U_{k}-1}q^{i-1},0,1)}{(1-c_{U_{k-1}+1}q^{i-1})\cdots(1-c_{U_{k}-1}q^{i-1})}\nonumber\\
&\qquad\cdot \prod_{m=n_{i-1}+1}^{k-1}\frac{\overline{\E}_{u_m}(c_{\overline{U_{m-1}+1}}q^{i-1},c_{U_{m-1}+1}q^{i-1},\ldots,c_{\overline{U_{m}}}q^{i-1},c_{U_{m}}q^{i-1})}{(1-c_{U_{n-1}+1}q^{i-1})\cdots(1-c_{U_{n}}q^{i-1})}\nonumber\\
&\qquad\cdot\prod_{j= 1}^{i-1}\frac{(c_j+c_{\overline{n_j}})q^j}{1-c_{n_j}q^{j}}\nonumber\\&\qquad\cdot \prod_{m=n_{j-1}+1}^{n_j-1}\frac{\overline{\E}_{u_m}(c_{\overline{U_{m-1}+1}}q^{j-1},c_{U_{m-1}+1}q^{j-1},\ldots,c_{\overline{U_{m}-1}}q^{j-1},c_{U_{m}}q^{j-1})}{(1-c_{U_{m-1}+1}q^{j-1})\cdots(1-c_{U_{m}}q^{j-1})}.
\label{eq:ovdus}
\end{align}
 
\end{theo}

For all $k,n\in \Zu$, set $U_{n,k}$ to be the sequence $U$ such that $u_{1+i(k+1)}=1$ and $u_{j+i(k+1)}=n$ for all $i\geq 0$ and $2\leq j\leq k+1$. For the sequence $U_{n,k}$, for all $i\geq 1$, set  $n_i=1+(i-1)(k+1)$, i.e. $U_{n_i} = 1+(i-1)(kn+1)$. Finally, set $c_{j+i(nk+1)}=c_{j}$ and $c_{\overline{j+i(nk+1)}}=c_{\overline{j}}$ for all $i\geq 0$ and $1\leq j\leq nk+1$. One can derive from Theorem \ref{theo:pus} the following result.

\begin{cor}
We have

\begin{align}
\sum_{\la\in \overline{\Pp}(U_{n,k})}C(\la)q^{|\la|_S}&= \frac{(-c_{\overline{1}}q;q)_\infty}{(c_{1}q;q)_\infty}\cdot \prod_{i=1}^k \frac{\prod_{j\geq 1} \overline{\E}_{n}(c_{\overline{2+(i-1)n}}q^{j},c_{2+(i-1)n}q^{j},\ldots,c_{\overline{in}}q^{j},c_{in}q^{j},c_{\overline{in+1}}q^{j},1)}{(c_{2+(i-1)n}q;q)_\infty \cdots (c_{1+in}q;q)_\infty},\\
\sum_{\la\in \Pp(U_{n,k})}C(\la)q^{|\la|_S}&= \frac{1}{(c_{1}q;q)_\infty}\cdot \prod_{i=1}^k \frac{\prod_{j\geq 1} \E_{n}(c_{2+(i-1)n}q^{j},\ldots,c_{in}q^{j})}{(c_{2+(i-1)n}q;q)_\infty \cdots (c_{1+in}q;q)_\infty},
\end{align}

\end{cor}
In particular, for $n=2$, we have 
\begin{align*}
\sum_{\la\in \overline{\Pp}(U_{2,k})}C(\la)q^{|\la|_S}&= \frac{(-c_{\overline{1}}q;q)_\infty}{(c_{1}q;q)_\infty}\cdot \prod_{i=1}^k \frac{(-c_{\overline{2i+1}}q,-(c_{2i}+2c_{\overline{2i}})q;q)_\infty}{(c_{2i}q,c_{2i+1}q;q)_\infty},\\
\sum_{\la\in \Pp(U_{2,k})}C(\la)q^{|\la|_S}&= \frac{1}{(c_{1}q;q)_\infty}\cdot \prod_{i=1}^k \frac{(-c_{2i}q;q)_\infty}{(c_{2i}q,c_{2i+1}q;q)_\infty},
\end{align*}
and we retrieve Theorem \ref{theo:ap2} from the latest identity. For $n=3$ and $c_{2+3i}=c_{3+3i}$ and $c_{\overline{2+3i}}=c_{\overline{3+3i}}$ for all $0\leq i\leq k-1$, we have 
\begin{footnotesize}
\begin{align*}
\sum_{\la\in \overline{\Pp}(U_{3,k})}C(\la)q^{|\la|_S}&= \frac{(-c_{\overline{1}}q;q)_\infty}{(c_{1}q;q)_\infty}\prod_{i=1}^k \frac{(-c_{\overline{3i+1}}q,-((3+\sqrt{3})c_{\overline{3i}}+(2+\sqrt{3})c_{3i})q,-((3-\sqrt{3})c_{\overline{3i}}+(2-\sqrt{3})c_{3i})q;q)_\infty}{(c_{3i}q,c_{3i}q,c_{3i+1}q;q)_\infty},\\
\sum_{\la\in \Pp(U_{3,k})}C(\la)q^{|\la|_S}&= \frac{1}{(c_{1}q;q)_\infty}\cdot \prod_{i=1}^k \frac{(-(2+\sqrt{3})c_{3i}q,-(2-\sqrt{3})c_{3i}q;q)_\infty}{(c_{3i}q,c_{3i}q,c_{3i+1}q;q)_\infty}.\\
\end{align*}
\end{footnotesize}
The remainder of the paper is organized as follows. In section \ref{sec:conjsw}, we investigate the link between the conjugate and the $S$-weight and prove Theorem \ref{theo:main} and Theorem \ref{theo:unrestrictedmodm}. Then, in Section \ref{sec:flatreg}, using our previous work related to a generalization of the Glaisher theorem, we provide a proof of Theorem \ref{theo:restrictedmodm}. Finally, in Section \ref{sec:compsympu}, we discuss the connection between the block over-partitions and the symmetric group and show Theorems \ref{theo:pu} and \ref{theo:pus}.


\section{Conjugate and $S$-weight}\label{sec:conjsw}


We discuss in this section the link between the conjugate and the $S$-weight of an over-partition. 

\subsection{The setup}

We first establish that the conjugacy is an involution on $\overline{\Pp}$.
\begin{lem}\label{lem:conjugate}
The map $\la \mapsto \tilde{\la}$ describes an weight-preserving involution on $\overline{\Pp}$ and on $\Pp$. Furthermore, it induces a bijection between $\overline{\Dd}$ and $\overline{\F}$.
\end{lem}
\begin{proof}
We first remark that a part $i\in \Zo \sqcup \overline{\Zu}$ has size $\sharp\{0<\overline{j}\leq i\}$. Let $\la=(\la_i)_{i\geq 1}$ be a sequence of $\Ppp$. We then have that 
\begin{align*}
|\la|&=\sum_{i\geq 1} \la_i \\
&= \sharp\{i,j\in \Zu: \la_{i}\geq \overline{j}\} \\
&=\sum_{j\geq 1} \tilde{\la}_j\\
&=|\tilde{\la}|.
\end{align*}
The map $\la\mapsto \tilde{\la}$ is then weight-preserving. Also, by definition, for all $i\in \Zu$, 
$$\sharp\{j\geq 1: \la_{j}=i,\overline{i}\} = \tilde{\la}_{i}-\tilde{\la}_{i+1}.$$ 
Moreover, as there is at most one occurrence of $\overline{i}$ in $\la$, 
$$\chi(\tilde{\la}_{i}\text{ is over-lined}) = \chi(\overline{i}\in \la) = \sharp\{j\geq 1: \la_{j}=\overline{i}\}.$$
Hence,
$$\sharp\{j\geq 1: \la_{j}=i\}= \tilde{\la}_{i}-\tilde{\la}_{i+1}-\chi(\tilde{\la}_{i}\text{ is over-lined}).$$
The sequence $\tilde{\la}$ is non-increasing in terms of part size, and we have that $\tilde{\la}_i>\tilde{\la}_{i+1}$ as soon as $\tilde{\la}_i$ is over-lined. The over-lined parts are then distinct, and then non-over-lined occurrences of a positive integer appear before a possible over-lined occurrence. Therefore, $\tilde{\la}\in \overline{\Pp}$, and by considering the conjugate of $\tilde{\la}$, for all $i\geq 1$, we have that
\begin{equation}\label{eq:conjugate1}
\sharp\{j\geq 1: \tilde{\la}_{j}=i,\overline{i}\} = \tilde{\tilde{\la}}_{i}-\tilde{\tilde{\la}}_{i+1} \qquad\text{and}\qquad \chi(\overline{i}\in \tilde{\la})=\chi(\tilde{\tilde{\la}}_i\text{ is over-lined}).
\end{equation}
Suppose now that $\la$ is an over-partition, i.e. $(\la_j)_j$ is non-increasing. For all $i\geq 1$, and for all $\la_{i+1}<\overline{j}\leq \la_i$, the part $\tilde{\la}_j$ has size $i$. In addition, $\tilde{\la}_j$ is over-lined if and only if $\overline{j}\in \la$, which, since $\la$ is non-increasing and $\la_{i+1}<\overline{j}\leq \la_i$, is equivalent to saying that $\overline{j}= \la_i$. Moreover, for $\overline{j}>\la_{i}$, $\tilde{\la}_j$ has size less than $i$, and for $\overline{j}\leq\la_{i+1}$, $\tilde{\la}_j$ has size greater than $i$.
Hence, 
\begin{equation}\label{eq:conjugate2}
\sharp\{j\geq 1: \tilde{\la}_{j}=i,\overline{i}\} = \la_i-\la_{i+1} \qquad\text{and}\qquad \chi(\overline{i}\in \tilde{\la})=\chi(\la_i\text{ is over-lined}).
\end{equation}
We note that $\tilde{\tilde{\la}}_j=0$ for all $\overline{j}>\tilde{\la}_1$. We thus conclude that, for all $0\leq \overline{i}\leq \tilde{\la}_1$, 
\begin{align*}
\tilde{\tilde{\la}}_i &= \sum_{j=i}^{\tilde{\la}_1+0} \tilde{\tilde{\la}}_{j}-\tilde{\tilde{\la}}_{j+1}\\
&= \sum_{j=i}^{\tilde{\la}_1+0} \sharp\{k\geq 1: \tilde{\la}_{k}=j,\overline{j}\}&\text{(by \eqref{eq:conjugate1})}\\
&= \sum_{j=i}^{\tilde{\la}_1+0} \la_{j}-\la_{j+1}&\text{(by \eqref{eq:conjugate2})}\\
&= \la_i,
\end{align*} 
and $\tilde{\tilde{\la}}_i$ is over-lined if and only if $\la_i$. Then, $\tilde{\tilde{\la}}_i=\la_i$ for all $i\in\Zu$, and  the conjugacy is an involution on $\overline{\Pp}$.
 
To prove that $\la \mapsto \tilde{\la}$ induces involution on $\Pp$, it suffices to observe that 
 there are as many over-lined parts in $\la$ as in $\Tilde{\la}$, so that $\tilde{\la}\in \Pp$ if and only if $\la \in \Pp$. 
 
 It now remains to show that the sets $\overline{\Dd}$  and $\overline{\F}$ are in bijection through the conjugacy. If $\la\in \overline{\Dd}$, then $\tilde{\la}\in \overline{\F}$ by \eqref{eq:conjugate2}, as $\overline{i}\in \tilde{\la}$ for all $0<\overline{i}\leq \tilde{\la}_1$. If $\la\in \overline{\F}$, then $\tilde{\la}\in \overline{\Dd}$ by \eqref{eq:conjugate1} with $\la$ instead of $\tilde{\la}$, as $\tilde{\la}_i>\tilde{\la}_{i+1}$ and $\tilde{\la}_i$ is over-lined for all  $0<\overline{i}\leq \tilde{\la}_1$.
\end{proof}

\subsection{Proof of Theorem \ref{theo:main}}\label{sec:main}

The association $i\mapsto(\lfloor i \rfloor_S)_{c_i}$ is a bijection from $\Zo\sqcup \overline{\Zu}$ to colored parts, as $c_i$ can only color the integer $\lfloor i \rfloor_S$ for all $i \in \Zo\sqcup \overline{\Zu}$. Consider the function 
\begin{align}
\Psi_S \colon \overline{\Pp} &\to  \overline{\Pp}^S_\C\nonumber\\
(\la_i)_{i\geq 1} &\mapsto ((S_{\la_i})_{c_{\la_i}})_{i\geq 1},\label{eq:psis}
\end{align}
Then, $\Psi_S$ is a bijection, and it induces bijection from $Ens$ to $Ens_\C^S$ for all subset $Ens$ of $\overline{\Pp}$.
Hence, as $\Phi_S(\la)=\Psi_S(\tilde{\la})$ for all $\la\in \overline{\Pp}$, along with Lemma \ref{lem:conjugate}, we have that $\Phi_S$ is a bijection from $\overline{\Pp}$ to  $\overline{\Pp}^S_\C$. Moreover, by definition, for all $\la=(\la_i)_{i\geq 1}\in \overline{\Pp}$, $$C(\Phi_S(\la))=\prod_{0<\overline{i}\leq \la_1}c_{\tilde{\la}_i}=C(\la).$$
In addition, as $\la$ is non-increasing, for all $i,j\geq 1$, $\la_i\geq \overline{j}$ if and only if $\overline{i}\leq \tilde{\la}_j$. Therefore, 
\begin{align*}
|\la|_S&= \sum_{i\in S} \la_i\cdot s_i\\
&= \sum_{i\geq 1} \sharp\{j\geq 1: \overline{i}\leq \tilde{\la}_j\}\cdot s_i\\
&= \sum_{j\geq 1} \sum_{\overline{i}\leq \tilde{\la}_j} s_i\\
&=\sum_{j\geq 1} |\tilde{\la}_j| &\text{(by \eqref{eq:sgif})}\\
&= |\Phi_S(\la)|.
\end{align*}
Finally, by Lemma \ref{lem:conjugate},
\begin{align*}
\Phi_S(\Pp)&=\Psi_S(\tilde{\Pp})=\Psi_s(\Pp)=\Pp_\C^S,\\
\Phi_S(\overline{\F})&=\Psi_S(\tilde{\overline{\F}})=\Psi_s(\overline{\Dd})=\overline{\Dd}_\C^S,\\
\Phi_S(\overline{\Dd})&=\Psi_S(\tilde{\overline{\Dd}})=\Psi_s(\overline{\F})=\overline{\F}_\C^S.\\
\end{align*}
The right-sum of \eqref{eq:ovp} is the generating function $\sum_{\la\in \overline{\Pp}_\C^S} C(\la)q^{|\la|}$, since for all $i\in \Zu$, the colored part $(S_i)_{c_{\overline{i}}}$ appears at most once, and $(S_i)_{c_{i}}$ can be repeated as many times as possible. Hence, \eqref{eq:ovp} stands true.
The same reasoning occurs for \eqref{eq:p} and \eqref{eq:ovf}. 
For \eqref{eq:ovd}, we compute the generating function of $\overline{\F}_\C^S$ by reasoning according to the minimal $\overline{i}$ greater than $j_1$, where $c_{j_1}=c(\la_1)$.
For $i\geq 1$,
$$ 
\prod_{0<\overline{j}<\overline{i}}
\frac{c_{\overline{j}} q^{S_j}}{1-c_jq^{S_j}}$$
is the generating function of the colored partitions $\la=(\la_t)_{t\geq 1}$ in $\overline{\F}_\C^S$ such that $\overline{i-1}\leq j_1<\overline{i}$. When $i=0$, the  empty product equals $1$ by convention, and this concludes the proof of Theorem \ref{theo:main}.

\subsection{Proof of Theorem \ref{theo:unrestrictedmodm}}\label{sec:modm}

Suppose that $s_{i+jm}=s_i$, $c_{i+jm}=c_{i}$ and $c_{\overline{i+jm}}=c_{\overline{i}}$ for all $j\geq 0$ and $1\leq i\leq m$. Hence,
for all $1\leq i\leq m$, for all $j\geq 0$, we have that 
$$S_{i+jm}=S_i + j\cdot S_m,$$ 
so that the part $i+jm$ is associated to the colored part $(S_i + j\cdot S_m)_{c_{i+jm}}$, and $\overline{i+jm}$ is associated to the colored part $(S_i + j\cdot S_m)_{c_{\overline{i+jm}}}$. 
In addition, by \eqref{eq:conjugate2}, for all $i\geq 1$,
$$\sharp\{j\geq 1: \tilde{\la}_{j}=i\} = \la_i-\la_{i+1}-\chi(\la_i\text{ is over-lined}) \qquad\text{and}\qquad \chi(\overline{i}\in \tilde{\la})=\chi(\la_i\text{ is over-lined}),$$ 
so that for all $1\leq i\leq m$,
$$\rho_i(\la) = \sum_{j\geq 0} \la_{i+mj}-\la_{i+1+mj} -\chi(\la_{i+mj} \text{ is over-lined}) = \sharp\{j\geq 1: \tilde{\la}_{j} \text{ is non-over-lined}, \,\,\tilde{\la}_{j}\equiv i \mod m\},$$
$$\rho_{\overline{i}}(\la) = \sharp\{j\geq 0: \la_{i+mj} \text{ is over-lined}\}= \sharp\{j\geq 1: \tilde{\la}_{j} \text{ is over-lined}, \,\,\tilde{\la}_{j}\equiv i \mod m\}.$$
Hence, 
$$C(\la) = \prod_{j=1}^m c_j^{\rho_j(\la)}c_{\overline{j}}^{\rho_{\overline{j}}(\la)}$$
and we deduce \eqref{eq:ovpmodm}, \eqref{eq:pmodm}, \eqref{eq:ovfmodm} and \eqref{eq:ovdmodm} respectively from \eqref{eq:ovp}, \eqref{eq:p}, \eqref{eq:ovf} and \eqref{eq:ovd}.


\section{Flat partitions, Regular partitions, and generalized Glaisher}\label{sec:flatreg}

In this session, we prove Theorem \ref{theo:restrictedmodm} using a refinement of a result of \cite{K22} generalizing the Glaisher identity via weighted words. To do so, we first present the generalization of the Glaisher theorem, and then suitably apply it to our context to prove Theorem \ref{theo:restrictedmodm}. 

\subsection{The setup}
Let $\C$ be a set of color, and let $\ep$ be a function from $\C^2$ to $\{0,1\}$. Suppose that there exists a color $c_0 \in \C$ such that, $\ep(c_0,c_0)=0$, and for all $c_0\neq c \in \C$, 
$$\ep(c,c_0)=0=1-\ep(c_0,c).$$
Consider the set $\mathbb{Z}_\C= \{k_c:k\in \mathbb{Z},c\in \C\}$ of colored integers, and define the binary relations $\gtrdot_\ep$ and $\gg_\ep$ on $\mathbb{Z}_\C$ by 
\begin{align}
k_c\gtrdot_\ep l_d &\text{ if and only if } k-l=\ep(c,d),\label{eq:flat}\\
k_c\gg_\ep l_d &\text{ if and only if } k-l\geq \ep(c,d).\label{eq:reg}
\end{align}

\begin{deff}
Let $\F_{\C,c_0}^{\ep}$ be the set of sequences $(\pi_0,\ldots,\pi_s)$ of colored parts well-ordered by $\gtrdot_\ep$, and such that $\pi_s=0_{c_0}$ and $\pi_{s-1}\neq 0_{c_0}$ when $s>0$. Such sequences are called flat partitions grounded in $c_0$ and with relation $\gtrdot_\ep$.
Let $\R_{\C,c_0}^{\ep}$ be the set of sequences $(\pi_0,\ldots,\pi_s)$ of colored parts well-ordered by $\gg_\ep$, and such that $\pi_s=0_{c_0}$ and $c(\pi_i)\neq c_0$ for all $0\leq i<s$. Such sequences are called regular partitions in $c_0$, grounded in $c_0$ and with relation $\gg_\ep$. Both types of partitions belongs to $\Pp_{\C,c_0}^{\ep}$, defined as the set of colored partitions $(\pi_0,\ldots,\pi_s)$ well-ordered by $\gg_\ep$, and such that $\pi_s=0_{c_0}$ and $\pi_{s-1}\neq 0_{c_0}$. The latter set is called the set of colored partitions grounded in $c_0$ and with relation $\gg_\ep$. The weight, color sequence and the arm of $\pi=(\pi_0,\ldots,\pi_s)\in \Pp_{\C,c_0}^{\ep}$ are respectively defined by 
$$|\pi|=\sum_{i=0}^{s-1}|\pi| \text{ , } C(\pi)=\prod_{i=0}^{s-1}c(\pi_i)\text{ and }\rho(\pi)=\pi_0-\sum_{i=0}^{s-1}\ep(c(\pi_i),c(\pi_{i+1})).$$
The set $\F_{\C,c_0}^{\ep}$ then corresponds to the subset of $\Pp_{\C,c_0}^{\ep}$ consisting of partitions $\pi$ with arm $0$, and 
the set $\R_{\C,c_0}^{\ep}$ then corresponds to the subset of $\Pp_{\C,c_0}^{\ep}$ consisting of partitions $\pi$ such that $C(\pi)$ does not have $c_0$ as factor.
\end{deff}

The generalization of Glaisher's identity in terms of colored partitions is the following.
\begin{theo}\label{theo:konan}
Let $n,o$ be non-negative integers and let $C$ be a sequence of colors in $\C\setminus\{c_0\}$. Then, the number of flat partitions $\pi$ in $\F_{\C,c_0}^{\ep}$ with weight $n$, with $o$ positive parts of color $c_0$ and such that the color sequence without $c_0$ is $C$ is equal to the number of regular partitions $\pi$ in $\R_{\C,c_0}^{\ep}$ with weight $m$, color sequence $C$ and arm $o$.
In terms of generating function, we then have that
$$\sum_{\pi\in \F_{\C,c_0}^{\ep}} C(\pi)q^{|\pi|}= \sum_{\pi\in \R_{\C,c_0}^{\ep}} c_0^{\rho(\pi)}C(\pi)q^{|\pi|}.$$
\end{theo} We now adjust the above result to fit our purpose. Let $m$ be a positive integer.
In the following, we set $\C=\{c_{i},c_{\overline{i+1}}:i\in \{0,\ldots,m-1\}\}$, and order the colors as follows:
$$c_0<c_{\overline{1}}<c_1<c_{\overline{2}}<\cdots<c_{\overline{m-1}}<c_{m-1}<c_{\overline{m}}.$$
Let $\ep$ be the function on $\C^2$ defined by 
$$\ep(c,d)=\chi(c<d)+\chi(c=d \text{ with an over-lined index}).$$
The full scope of the matrix $M=(\ep(c_i,c_j))_{0\leq i,j\leq \overline{m}}$ is given by
$$\bordermatrix{
\text{}&0& \overline{1} &1&\cdots & \overline{m-1}&m-1&\overline{m}
\cr 0&0&1&1&\cdots&1&1&1
\cr \overline{1}&0&1&1&\cdots&1&1&1
\cr 1&0&0&0&\ddots&1&1&1
\cr \vdots&\vdots&\vdots&\ddots&\ddots&\ddots&\vdots&\vdots
\cr \overline{m-1}&0&0&0&\ddots &1&1&1
\cr m-1&0&0&0&\cdots&0&0&1
\cr \overline{m}&0&0&0&\cdots&0& 0&1
} .$$
Then, $\ep(c_0,c_0)=0$, and for all $c_0\neq c \in \C$, 
$$\ep(c,c_0)=0=1-\ep(c_0,c).$$
The relation $\gg_\ep$ also induces an order on colored parts
$$\cdots\gg_\ep 1_{c_0}\gg_\ep 1_{c_0}\gg_\ep 0_{c_{\overline{m}}}\gg_\ep 0_{c_{m-1}}\gg_\ep0_{c_{m-1}}\gg_\ep 0_{c_{\overline{m-1}}}\gg_\ep \cdots \gg_\ep 0_{c_{\overline{2}}}\gg_\ep 0_{c_1}\gg_\ep 0_{c_1}\gg_\ep 0_{c_{\overline{1}}}\gg_\ep 0_{c_0},$$
so that the parts with a color with an over-lined index are distinct. 
\[\]
Suppose that $s_{i+jm}=s_i$, $c_{i-1+jm}=c_{i-1}$ and $c_{\overline{i+jm}}=c_{\overline{i}}$ for all $j\geq 0$ and $1\leq i\leq m$. We then have that $$\C = \{c_0,c_{\overline{1}},c_1,\ldots,c_{\overline{m-1}},c_{m-1},c_{\overline{m}}\} = \{c_i,c_{\overline{i+1}}:i\in \Zo\}.$$
Define the function $\alpha_S$ on $(\Zo)_{\C}$ such that
\begin{align*}
j_{c_i} &\mapsto (jS_m+S_i)_{c_{jm+i}} &\text{ for all } 0\leq i\leq m-1 \\
j_{c_{\overline{i}}} &\mapsto (jS_m+S_i)_{c_{\overline{jm+i}}} &\text{ for all } 1\leq i\leq m.
\end{align*}
By abuse of notation, we extend $\alpha_S$ to the set $\Pp_{\C,c_0}^{\ep}$ with
\begin{align*}
\alpha_S \colon \Pp_{\C,c_0}^\ep &\to  \overline{\Pp}^S_\C\nonumber\\
(\pi_0,\ldots,\pi_s) &\mapsto (\la_i)_{i\geq 1}
\end{align*}
where $\la_{i+1}=\alpha_S(\pi_i)$ for all $0\leq i\leq s$ and $\la_i=0$ for all $i>s+1$. Hence, $\alpha_S$ describes a bijection from $\Pp_{\C,c_0}^{\ep}$ to $\overline{\Pp}_\C^S$. In terms of generating functions, 
it is equivalent to doing the transformations
\begin{align*}
q &\mapsto  q^{S_m}\\
c_i &\mapsto c_iq^{S_i} &\text{ for all } 0\leq i\leq m-1 \\
c_{\overline{i}} &\mapsto c_{\overline{i}}q^{S_i} &\text{ for all } 1\leq i\leq m.
\end{align*}
Then, by setting the arm of $\la \in \overline{\Pp}_\C^S$ to be the arm of $\alpha_S((\Psi_S^{-1}(\la))$ as in \eqref{eq:psis}, Theorem \ref{theo:konan} then stands true for the sets corresponding to $\F_{\C,c_0}^{\ep}$ and $\R_{\C,c_0}^{\ep}$. It is easy to see that $\alpha_S(\R_{\C,c_0}^{\ep})$ is the set of over-partitions with non-null parts not colored by $c_0$. For $\F_{\C,c_0}^{\ep}$, we have for integers $i,j$
\begin{align*}
k_{c_i}\gtrdot l_{c_j}&\Longleftrightarrow k-l =\chi(i<j)\\
&\Longleftrightarrow (i+km)-(j+lm)=i-j+m\chi(i<j),\\
k_{c_i}\gtrdot l_{c_{\overline{j}}}&\Longleftrightarrow k-l =\chi(i<j)\\
&\Longleftrightarrow (i+km)-(j+lm)=i-j+m\chi(i<j),\\
k_{c_{\overline{i}}}\gtrdot l_{c_{j}}&\Longleftrightarrow k-l =\chi(i<j)\\
&\Longleftrightarrow (i+km)-(j+lm)=i-j+m\chi(i<j),\\
k_{c_{\overline{i}}}\gtrdot l_{c_{\overline{j}}}&\Longleftrightarrow k-l =\chi(i\leq j)\\
&\Longleftrightarrow (i+km)-(j+lm)=i-j+m\chi(i\leq j).
\end{align*}
Hence, the set $\alpha_S(\F_{\C,c_0}^{\ep})$ consists of over-partitions $\la=(\la_i)_{i\geq 1}\in \overline{\Pp}_\C^S$  such that, by setting $\mu=(\mu_i)_{i\geq 1}= \Psi_S^{-1}(\la)$ as in \eqref{eq:psis}, we have $\mu_i -\mu_{i+1}\leq m-1+\chi(\mu_i \text{ is over-lined})$. 
We thus have the following corollary of Theorem \ref{theo:konan}.

\begin{cor}\label{cor:konan}
Let $m,o$ be non-negative integers and let $C$ be a sequence of colors in $\C\setminus\{c_0\}$. Then, the number of over-partitions $\la=(\la_i)_{i\geq 1}\in \overline{\Pp}_\C^S$  such that, by setting $\mu=(\mu_i)_{i\geq 1}= \Psi_S^{-1}(\la)$ as in \eqref{eq:psis}, we have $\mu_i-\mu_{i+1}\leq m-1+\chi(\mu_i \text{ is over-lined})$, with weight $m$, with $o$ non-null parts with color $c_0$ and such that the color sequence without $c_0$ is $C$, is equal to the number of over-partitions of $\overline{\Pp}_\C^S$ with the non-null parts not colored by $c_0$, with weight $m$, arm $o$ and color sequence $C$.
\end{cor}

\subsection{Proof of Theorem \ref{theo:restrictedmodm}}\label{sec:rmodm}

Let $m$ be a positive integer. Recall that $\overline{\Pp}_m$ is the set of over-partitions with less than $m$ non-over-lined occurrences for all positive integers,
$\Pp_m$ is the set of partitions of $\overline{\Pp}_m$, and $\overline{\F}_m$ is the subset of $\overline{\Pp}_m$ consisting of over-partitions such that $\la_i-\la_{i+1}=\chi(\la_i \text{ is over-lined})$ for all $i$ not divisible by $m$.
The next lemma gives us a precise description of the image of $\overline{\Pp}_m,\Pp_m$ and $\overline{\Dd}_m$ via the conjugacy.

\begin{lem}\label{lem:conjugatemodm}
The map $\la\mapsto \tilde{\la}$ describes a bijection:
\begin{enumerate}
\item from $\overline{\Pp}_m$ to the set of over-partitions $(\mu_i)_{i\geq 1}$ such that $\mu_i-\mu_{i+1}\leq m-1+\chi(\mu_i \text{ is over-lined})$,
\item from $\Pp_m$ to the set of partitions $(\mu_i)_{i\geq 1}$ such that $\mu_i-\mu_{i+1}\leq m-1$,
\item from $\overline{\F}_m$ to the set of over-partitions $(\mu_i)_{i\geq 1}$ of $\overline{\Pp}_m$ such that the parts not divisible by $m$ are over-lined.
\end{enumerate} 
\end{lem}
\begin{proof}
The facts (1) and (2) are straightforward, as by \eqref{eq:conjugate1}, 
$$\tilde{\la}_i-\tilde{\la}_{i+1} = \sharp\{j\geq 1:\la_j=i\}+\chi(\overline{i}\in \la)\text{ and }\chi(\overline{i}\in \la) =\chi(\tilde{\la}_i \text{ is over-lined}).$$
Similarly, the fact (3) derives from \eqref{eq:conjugate2}, as
$$\la_i-\la_{i+1} = \sharp\{j\geq 1:\tilde{\la}_j=i\}+\chi(\overline{i} \in \tilde{\la})\text{ and }\chi(\overline{i}\in \tilde{\la})=\chi(\la_i \text{ is over-lined})$$
and, in $\overline{\Pp}_m$, being in $\overline{\F}_m$ is equivalent to saying that $$\la_i-\la_{i+1}=\chi(\la_i \text{ is over-lined})$$ for all $i$ not divisible by $m$.
\end{proof}
For $c_m=c_0=1$, by using Corollary \ref{cor:konan} and Lemma \ref{lem:conjugatemodm}, there exists a bijection between
\begin{enumerate}
\item $\overline{\Pp}_m$ and the set of colored partitions of $\overline{\Pp}_\C^S$ with no positive multiple of $t$ colored by $c_m$,
\item $\Pp_m$ and the set of colored partitions of $\Pp_\C^S$ with no positive multiple of $t$ colored by $c_m$,
\item $\overline{\F}_m$ and the set of colored partitions of $\overline{\Pp}_\C^S$ with no part with non-over-lined index's color different from $0_{c_0}$.
\end{enumerate}
Moreover, for $\la$ maps to $\mu$, we have $C(\la)=C(\mu)$ and $|\la|_S = |\mu|$. This then yields to \eqref{eq:restrictedovpmodm},\eqref{eq:restrictedpmodm}, and \eqref{eq:restrictedovdmodm}.


\section{Symmetric group and block over-partitions}\label{sec:compsympu}


In \cite{St11}, Stanley discusses the link between the finite sequences of non-negative integers (compositions) and the symmetric group.  We here adjust this link to describe the block over-partitions and related it to the over-partitions.

\subsection{The setup}

Let $u$ be a positive integer. Let $\Ppp^{(u)}$ be the set of sequences of $u$ terms in $\Zo\sqcup \overline{\Zu}$ such that each over-lined part appears once and after the non-over-lined parts of the same size. Such a sequence is called \textit{over-composition}. Recall $\mathcal{S}_u$ the set of permutations of $\{1,\ldots,u\}$.  

\begin{deff}\label{def:sadmissible}
Let $\sigma$ be a permutation in $\mathcal{S}_u$. Let $\overline{\Pp}^{(u)}$ be the set of over-compositions of $\Ppp^{(u)}$ which are non-increasing.
We say that $\mu\in\overline{\Pp}^{(u)}$ is $\sigma$-admissible if $\mu_{i}-\mu_{i+1}>0$ for all $1\leq i\leq u-1$ such that $\sigma(i)>\sigma(i+1)$.
Let $\Sigma_u$ to be the set of pairs $(\sigma,\mu)\in \mathcal{S}_u\times \overline{\Pp}^{(u)}$ such that $\mu$ is $\sigma$-admissible.
\end{deff}

Let $\la=(\la_1,\ldots,\la_u)\in \Ppp^{(u)}$. There then exist a unique permutation $\sigma$ of $\{1,\ldots,u\}$ such that, for all $1\leq i<j\leq u$, 
\begin{equation}\label{eq:sadmissible}
\sigma^{-1}(i)<\sigma^{-1}(j)\text{ if and only if }\la_{i}-\la_{j}\geq 0.
\end{equation}
To be more precise, we have for all $1\leq i\leq u$ that
\begin{align}
\sigma^{-1}(i)&=\sharp\{j\leq i:\la_j-\la_i\geq 0\}+\sharp\{j>i:\la_j-\la_i>0\} \nonumber\\
&= \sharp\{j:\la_j-\la_i>0\}+\sharp\{j\leq i:\la_j-\la_i=0\},\label{eq:compsigma}
\end{align}
Hence, $(\la_{\sigma(1)},\ldots,\la_{\sigma(u)})$ is non-increasing in terms of size. Moreover, if $i<j$ and $\la_i-\la_j=0$, then $\sigma^{-1}(i)<\sigma^{-1}(j)$. Therefore, the fact that an over-lined part appears after the non-over-lined parts of the same size implies that  $(\la_{\sigma(1)},\ldots,\la_{\sigma(u)})\in \overline{\Pp}^{(u)}$.

\begin{lem}\label{lem:compsigma}
Define the map 
\begin{align*}
\Gamma_u \colon \overline{\Pp}^{(u)} &\to  \Sigma_u\\
(\la_1,\ldots,\la_u) &\mapsto (\sigma,(\la_{\sigma(1)},\ldots,\la_{\sigma(u)})),
\end{align*}
where $\sigma$ satisfies \eqref{eq:compsigma}.
The map $\Gamma_u$ is then well-defined, and it describes a bijection from $\Ppp^{(u)}$ to $\Sigma_u$. 
\end{lem}

\begin{proof}
Let $\la\in \Ppp^{(u)}$. Then, $\sigma$ is such that $(\la_{\sigma(1)},\ldots,\la_{\sigma(u)})\in \overline{\Pp}^{(u)}$. Let $1\leq t\leq u-1$, and suppose that $\sigma(t)>\sigma(t+1)$. Then, applying \eqref{eq:sadmissible} with $i=\sigma(t+1)$ and $j=\sigma(t)$, we obtain that $$\sigma^{-1}(i)=t+1>t=\sigma^{-1}(j),$$ and $\la_{\sigma(t+1)}-\la_{\sigma(t)}=\la_{j}-\la_{i}<0$.
Hence, $(\la_{\sigma(1)},\ldots,\la_{\sigma(u)})$ satisfies that
$\la_{\sigma(i)}-\la_{\sigma(i+1)}>0$ for all $1\leq i\leq u-1$ such that $\sigma(i)>\sigma(i+1)$. Therefore, $(\la_{\sigma(1)},\ldots,\la_{\sigma(u)})$ is $\sigma$-admissible and $\Gamma_u$ is well-defined.

It is straightforward to see that $\Gamma_u$ is injective, as a pair $(\sigma,(\mu_1,\ldots,\mu_u))$ can only have as pre-image
$(\mu_{\sigma^{-1}(1)},\ldots,\mu_{\sigma^{-1}(u)})$. It then remains to show that $\Gamma_u$ is surjective. Let $(\sigma,(\mu_1,\ldots,\mu_u))\in \Sigma_u$. Note that for all $1\leq i\leq u-1$,
$\mu_i-\mu_{i+1}\geq 0$. Hence, if
$\mu_i-\mu_{i+1}=0$, then $\sigma(i)<\sigma(i+1)$. For $1\leq i<j\leq u$, if $\mu_j-\mu_i=0$, then $\mu_{t}-\mu_{t+1}=0$ for all $i\leq t<j$ and 
$$\sigma(i)<\sigma(i+1)<\cdots<\sigma(j)$$ so that $\sigma(i)<\sigma(j)$.
Thus, for all $1\leq j\leq u$,
$$\sharp\{i: \sigma(i)\leq \sigma(j), \mu_i-\mu_j=0\} = \sharp\{i\leq  j: \mu_i-\mu_j=0\}.$$
An over-lined part occurs after the non-over-lined parts of same size in $(\mu_1,\ldots,\mu_u)$, so that it is also the case in $(\mu_{\sigma^{-1}(1)},\ldots,\mu_{\sigma^{-1}(u)})$. Hence, $(\mu_{\sigma^{-1}(1)},\ldots,\mu_{\sigma^{-1}(u)})\in \Ppp^{(u)}$.
Finally, for all $1\leq i\leq u$, 
\begin{align*}
\sharp\{j:\mu_{\sigma^{-1}(j)}-\mu_{\sigma^{-1}(i)}>0\}+\sharp\{j\leq i:\mu_{\sigma^{-1}(j)}-\mu_{\sigma^{-1}(i)}=0\}&=\sharp\{j:\mu_{j}-\mu_{\sigma^{-1}(i)}>0\}\\
&\qquad+\sharp\{j:\sigma(j)\leq i, \,\mu_{j}-\mu_{\sigma^{-1}(i)}=0\}\\
&=\sharp\{j:\mu_{j}-\mu_{\sigma^{-1}(i)}>0\}\\
&\qquad+\sharp\{j\leq\sigma^{-1}(i) :\mu_{j}-\mu_{\sigma^{-1}(i)}=0\}\\
&=\sharp\{j\leq \sigma^{-1}(i):\mu_{j}- \mu_{\sigma^{-1}(i)}\geq 0\}\\
&=\sigma^{-1}(i).
\end{align*}
By \eqref{eq:compsigma}, $\Gamma_u((\mu_{\sigma^{-1}(1)},\ldots,\mu_{\sigma^{-1}(u)}))=(\sigma,(\mu_1,\ldots,\mu_u))$, and we conclude. 
\end{proof}

We here adopt the aforementioned tool to state a result analogous to Lemma \ref{lem:compsigma} for the block over-partitions. Recall that $U=(u_n)_{n\geq 1}$ is a sequence of positive integers and that $U_n = \sum_{i=1}^n u_n$ for all $n\geq 0$. 

\begin{deff}\label{eq:Sadmissible}
For a positive integer $u$, let $e_u$ be the identity permutation of $\{1,\ldots,u\}$.
Let $\Sigma=(\sigma_i)_{i\geq 1} \in \prod_{i\geq 1} \mathcal{S}_{u_i}$ be  a sequence of permutations ultimately equal to the sequence $(e_{u_i})_{i\geq 1}$.
Let $\la =(\la_j)_{j\geq 1}\in \overline{\Pp}$. 
We say that $\la$ is $\Sigma$-admissible if $(\la_{U_{i-1}+1},\ldots,\la_{U_i})\in \overline{\Pp}^{(u)}$ is $\sigma_i$-admissible for all $i\geq 1$, i.e. $\mu_{j}-\mu_{j+1}>0$ for all $U_{i+1}+1\leq j< U_i$ such that $\sigma_i(j-U_{i-1})>\Sigma(j+1-U_{i-1})$.
Let $\Sigma_U$ to be the set of pairs $(\Sigma,\la)$ such that $\mu$ is $\Sigma$-admissible.
\end{deff}

For all $\Sigma=(\sigma_i)_{i\geq 1} \in \prod_{i\geq 1} \mathcal{S}_{u_i}$ ultimately equal to the sequence $(e_{u_i})_{i\geq 1}$, set for all $i\geq 1$ and for all $U_{i-1}+1\leq j\leq U_i$
$$\Sigma(j) = U_{i-1}+\sigma_i(j-U_{i-1}).$$
Hence, for all $i\geq 1$ and for all $U_{i+1}+1\leq j< U_i$, $\sigma_i(j-U_{i-1})>\sigma_i(j+1-U_{i-1})$ if and only if $\Sigma(j)>\Sigma(j+1)$. Moreover, for all $i\geq 1$, 
$$\Sigma(U_i)=U_{i-1}+\sigma_i(u_i)\leq U_i<U_i+1\leq U_i+\sigma_{i+1}(1)=\Sigma(U_i+1).$$
Therefore, for all $\la\in \overline{\Pp}$, 
\begin{equation}\label{eq:equivalencesigma}
\la \text{ is }\Sigma\text{-admissible if and only }\la_j-\la_{j+1}>0\text{ for all }j\geq 1 \text{ such that }\Sigma(j)>\Sigma(j+1).
\end{equation}
We then have the following.
\begin{lem}\label{lem:compSigma}
Define the map 
\begin{align*}
\Gamma_U \colon \overline{\Pp}(U) &\to \Sigma_U\\
\la &\mapsto (\Sigma,\tilde{\tilde{\la}}),
\end{align*}
where $\Sigma=(\sigma_i)_{i\geq 1}$ such that, for all $i\geq 1$,
$$\Gamma_{u_i}((\la_{U_{i-1}+1},\ldots,\la_{U_i}))=(\sigma_i, \la_{U_{i-1}+\sigma_i(1)},\ldots,\la_{U_{i-1}+\sigma_i(u_i)})).$$
The map $\Gamma_U$ is then well-defined, and it describes a bijection from $\overline{\Pp}(U)$ to $\Sigma_U$. 
\end{lem}
\begin{proof}
Consider the map 
\begin{align*}
\Lambda_U \colon \overline{\Pp}(U) &\to \Sigma_U\\
\la &\mapsto (\Sigma,\mu),
\end{align*}
where $\Sigma=(\sigma_i)_{i\geq 1}$ and $\mu=(\mu_i)_{i\geq 1}$ are such that, for all $i\geq 1$,
$$\Gamma_{u_i}((\la_{U_{i-1}+1},\ldots,\la_{U_i}))=(\sigma_i,(\mu_{U_{i-1}+1},\ldots,\mu_{U_i})).$$
Note the map $\Lambda_U$ is injective as $\Gamma_{u_i}$ is injective for all $i\geq 1$. Since $\la$ is a finite number of positive terms and $(U_j)_{j\geq 1}$ is an increasing sequence of positive integers, there exists $\ell\geq 1$ such that $\la_j=0$ for all $j>U_{\ell-1}$. Hence, for all $j\geq \ell$,
$$(\la_{U_{j-1}+1},\ldots,\la_{U_j})=(\mu_{U_{j-1}+1},\ldots,\mu_{U_j})=(0,\ldots,0).$$
and $\sigma_j=e_{u_j}$. Moreover, by Lemma \ref{lem:compsigma},
for all $i\geq 1$,
$(\mu_{U_{i-1}+1},\ldots,\mu_{U_i})$ is non-increasing and $\sigma_i$-admissible, 
$$\max\{\la_{U_{i-1}+1},\ldots,\la_{U_i}\}=\mu_{U_{i-1}+1}$$
and
$$\min\{\la_{U_{i-1}+1},\ldots,\la_{U_i}\}=\mu_{U_i}.$$
Hence, as $\la\in\overline{\Pp}(U)$, we then have that $\mu_{U_{i}}\geq \mu_{U_{i}+1}$ for all $i\geq 1$. Along with the fact that the over-lined parts are distinct, we deduce that  $\mu$ is an over-partition and $(\Sigma,\mu)\in \Sigma_U$. Inversely, for $(\Sigma,\mu)\in \Sigma_U$, since $\Gamma_{u_i}$ is surjective for all $i\geq 1$, one can define 
$\la=(\la_i)_{i\geq 1}$ such that 
$$\Gamma_{u_i}((\la_{U_{i-1}+1},\ldots,\la_{U_i}))=(\sigma_i,(\mu_{U_{i-1}+1},\ldots,\mu_{U_i})).$$
By reversing the previous reasoning, we have that $\la$ has a finite number of positive terms, and 
$$\min\{\la_{U_{i-1}+1},\ldots,\la_{U_i}\}\geq \max\{\la_{U_{i}+1},\ldots,\la_{U_{i+1}}\}$$
so that $\la \in \overline{\Pp}(U)$.
Hence, $\Lambda_U$ describes a bijection from $\overline{\Pp}(U)$ to $\Sigma_U$.

Now, observe that $\mu$ is obtained by ordering the parts of $\la$ in non-increasing order. Then, for all $i\geq 1$,
$$\tilde{\la}_i = \sharp\{j\geq 1: \la_j\geq \overline{i}\} = \sharp\{j\geq 1: \mu_j\geq \overline{i}\} = \tilde{\mu}_i$$
and $\tilde{\la}=\tilde{\mu}$. Since $\mu$ is an over-partition, we then have  by Lemma \ref{lem:conjugate} that $\mu = \tilde{\tilde{\mu}}=\tilde{\tilde{\la}}$.
Therefore, $\Gamma_U=\Lambda_U$ and we conclude.
\end{proof}
We conclude this session by defining a notion equivalent to the $\Sigma$-admissibility.
\begin{deff}
Let $\Sigma=(\sigma_i)_{i\geq 1} \in \prod_{i\geq 1} \mathcal{S}_{u_i}$ be  a sequence of permutations ultimately equal to the sequence $(e_{u_i})_{i\geq 1}$, and let $\la$ be an over-partition. We say that $\la$ is $\Sigma$-sizable if $\la$ has a part of size $i$ for all $i\geq 1$ such that $\Sigma(i)>\Sigma(i+1)$.
Let $\Omega_U$ be the set of pairs $(\Sigma,\la)$ such that $\la$ is $\Sigma$-sizable.
\end{deff}
We then have the following equivalence.
\begin{lem}\label{lem:compOmega}
The map 
\begin{align*}
\Sigma_U &\to \Omega_U\\
(\Sigma,\la) &\mapsto (\Sigma,\tilde{\la})
\end{align*}
\end{lem}
is a bijection.
\begin{proof}
By \eqref{eq:conjugate2}, 
$\la_{i}-\la_{i+1}>0$  if and only if $i \text{ or }\overline{i} \in \tilde{\la}$ and we conclude.
\end{proof}

\subsection{Proof of Theorem \ref{theo:pu}}\label{sec:pu}

From Lemmas \ref{lem:compSigma} and \ref{lem:compOmega}, the map $\la \mapsto (\Sigma,\tilde{\la})$ from $\overline{\Pp}(U)$ to $\Omega_U$, with $(\Sigma,\tilde{\tilde{\la}}) = \Gamma_U(\la)$, is a bijection. Hence,
$$\sum_{\la\in \overline{P}(U)} C(\la)q^{|\la|} = \sum_{(\Sigma,\la)\in \Omega_U} (\prod_{\substack{i\geq 1\\ \la_i>0}}c_{\la_i})q^{|\la|}.$$
For a fixed $\Sigma=(\sigma_i)_{i\geq 1}$
ultimately equal to $(e_{u_i})_{i\geq 1}$, we have that
$$\sum_{(\Sigma,\la)\in \Omega_U} (\prod_{\substack{i\geq 1\\ \la_i>0}}c_{\la_i})q^{|\la|} = \prod_{i\geq 1} \frac{(c_{i}q^i)^{\chi(\Sigma(i)>\Sigma(i+1)}+c_{\overline{i}}q^i)}{1-c_{i}q^i},$$
the generating function of parts of size $i$
being $\frac{1+c_{\overline{i}}q^i}{1-c_{i}q^i}$ when unrestricted, and $\frac{c_{i}q^i+c_{\overline{i}}q^i}{1-c_{i}q^i}$ when there is at least one part of size $i$. For $\ell\geq 1$, we now sum the latter equality over the sequences of permutations $\Sigma=(\sigma_i)_{i\geq 1}$ such that $\sigma_i=e_{u_i}$ for all $i\geq\ell$. Then, by \eqref{eq:equivalencesigma},
\begin{align*}
\sum_{\substack{\Sigma\\
\sigma_i=e_{u_i}\\ \forall \,i\geq \ell}}\sum_{(\Sigma,\la)\in \Omega_U} (\prod_{\substack{i\geq 1\\ \la_i>0}}c_{\la_i})q^{|\la|} &= \prod_{i\geq 1} \frac{1+(c_{\overline{i}}q^i)^{\chi(i>U_{\ell-1})}}{1-c_{i}q^i}\sum_{\substack{\Sigma\\
\sigma_i=e_{u_i}\\ \forall \,i\geq \ell}} \prod_{i=1}^{\ell-1}\prod_{j=U_{i-1}+1}^{U_i}((c_{j}q^j)^{\chi(\Sigma(j)>\Sigma(j+1)}+c_{\overline{j}}q^j)\\
&= \prod_{i\geq 1} \frac{1+(c_{\overline{i}}q^i)^{\chi(i>U_{\ell-1})}}{1-c_{i}q^i}\\
&\quad\cdot \sum_{\sigma_1,\ldots,\sigma_{\ell-1}} \prod_{i=1}^{\ell-1}(1+c_{U_i}q^i)\prod_{j=U_{i-1}+1}^{U_i-1}((c_{j}q^j)^{\chi(\sigma_i(j-U_{i-1})>\sigma_j(j+1-U_{i-1})}+c_{\overline{j}}q^j)\\
&= \prod_{i\geq 1} \frac{1+(c_{\overline{i}}q^i)^{\chi(i>U_{\ell-1})}}{1-c_{i}q^i}\\
&\quad\cdot \prod_{i=1}^{\ell-1} (1+c_{U_i}q^i)\sum_{\sigma_i\in \mathcal{S}_{u_i}} \prod_{j=U_{i-1}+1}^{U_i-1}((c_{j}q^j)^{\chi(\sigma_i(j-U_{i-1})>\sigma_j(j+1-U_{i-1})}+c_{\overline{j}}q^j)
\\
&= \prod_{i\geq 1} \frac{1+(c_{\overline{i}}q^i)^{\chi(i>U_{\ell-1})}}{1-c_{i}q^i}\\
&\quad\cdot \prod_{i=1}^{\ell-1} \overline{\E}_{u_i}(c_{\overline{U_{i-1}+1}}q^{U_{i-1}+1},\ldots,c_{\overline{U_{i}-1}}q^{U_{i}-1},c_{U_{i}-1}q^{U_{i}-1},c_{\overline{U_{i}}}q^{U_{i}},1).
\end{align*}
Therefore, when $\ell$ tends to $\infty$, we obtain \eqref{eq:ovpu}.
For \eqref{eq:pu}, in the above proof, we sum over the $(\Sigma,\la)\in \Omega_U$ such that $\la \in \Pp$. This is equivalent to saying that there are no over-lined parts, and the generating function then consists of that of $(\Sigma,\la)\in \Omega_U$ with $c_{\overline{i}}=0$ for all $i\geq 1$. Moreover, by Lemma \ref{lem:conjugate}, the map $\la\mapsto (\Sigma,\tilde{\la})$ induces a bijection between $\Pp(U)$ and the pairs of $\Omega_U$ with the second terms being a partition. Hence, \eqref{eq:pu} stands true. 
Finally, for \eqref{eq:ovdu}, we first observe that the  $\la\mapsto (\Sigma,\tilde{\la})$ induces a bijection between $\Pp(U)$ and the pairs of $\Omega_U$ such that the second terms is an over-partition $\mu$ with parts of size $U_i$ for all $0<\overline{U_i}\leq \mu_1$. 
For all $i\geq 1$, the generating function of such pairs satisfying $-u_i\leq \mu_1-U_i<0$ is given by the expression
\begin{align*}
&\frac{\overline{\E}_{u_i}(c_{\overline{U_{i-1}+1}}q^{U_{i-1}+1},c_{U_{i-1}+1}q^{U_{i-1}+1},\ldots,c_{\overline{U_{i}-1}}q^{U_{i}-1},c_{U_{i}-1}q^{U_{i}-1},0,1)}{(1-c_{U_{i-1}+1}q^{U_{i-1}+1})\cdots(1-c_{U_{i}-1}q^{U_{n}-1})}\\
&\qquad\cdot \prod_{j=1}^{i-1}\frac{\overline{\E}_{u_j}(c_{\overline{U_{j-1}+1}}q^{U_{j-1}+1},c_{U_{j-1}+1}q^{U_{j-1}+1},\ldots,c_{\overline{U_{j}}}q^{U_{j}},c_{U_{j}}q^{U_{j}})}{(1-c_{U_{j-1}+1}q^{U_{j-1}+1})\cdots(1-c_{U_{j}}q^{U_{j}})}.
\end{align*}
Summing over $i\geq 1$ yields to \eqref{eq:ovdu}.
\subsection{Proof of Theorem \ref{theo:pus}}\label{sec:pus}

Let $\la \in \overline{\Pp}(U)$ and $\Gamma_U(\la) = (\Sigma,\tilde{\tilde{\la}})$.
Since $u_{n_i}=1$ for all $i\geq 1$, we then have that $\sigma_{n_i}=e_{1}$ the unique element of $\mathcal{S}_1$, so that $\tilde{\tilde{\la}}_{s_i}=\la_{\Sigma(U_{n_i})}=\la_{U_{n_i}}$. Hence,
$$\sum_{\la\in \overline{\Pp}(U)}C(\la)q^{|\la|_S} = \sum_{\la\in \overline{\Pp}(U)}C(\la)q^{|\tilde{\tilde{\la}}|_S}
=\sum_{(\Sigma,\la)\in \Omega_U}(\prod_{\substack{i\geq 1\\ \la_i>0}}c_{\la_i})q^{\sum_{i\geq 1} S_{\la_i}}$$
This is obtained by replacing $c_iq^{i}$ and $c_{\overline{i}}q^i$ respectively by $c_iq^{S_i}$ and $c_{\overline{i}}q^{S_i}$ in the right-sum of \eqref{eq:ovpu}. Finally, $S_j=i-1$ for all $U_{n_{i-1}}\leq j<U_{n_i}$ and we deduce to \eqref{eq:ovpus}. Similarly, we deduce \eqref{eq:pus} and \eqref{eq:ovdus} respectively from \eqref{eq:pu} and \eqref{eq:ovdu}.

\end{document}